%% file: survey-cc.tex
\documentclass[12pt]{amsart}
\usepackage{verbatim}
\usepackage[dvips]{color}
\usepackage{amssymb}
\usepackage[active]{srcltx}

\usepackage{latexsym}
\usepackage{amsmath}
\usepackage{float}
\usepackage{mathrsfs}
\usepackage{esint}
\usepackage{enumerate}

\usepackage{amsmath, amssymb, amsthm,  epsfig}
\usepackage{times}

\theoremstyle{plain}
\newtheorem{thrm}{Theorem}[section]
\newtheorem*{thrm*}{Theorem}
\newtheorem{lemma}[thrm]{Lemma}
\newtheorem{prop}[thrm]{Proposition}
\newtheorem{cor}[thrm]{Corollary}
\theoremstyle{definition}
\newtheorem{dfn}[thrm]{Definition}
\theoremstyle{remark}
\newtheorem{rmrk}[thrm]{Remark}
\theoremstyle{example}
\newtheorem{ex}[thrm]{Example}
\numberwithin{equation}{section}
\usepackage{color}

\begin{document}

\newcommand{\tx}{\tilde x}
\newcommand{\R}{\mathbb R}
\newcommand{\N}{\mathbb N}
\newcommand{\C}{\mathbb C}
\newcommand{\lie}{\mathcal G}
\newcommand{\hN}{\mathcal N}
\newcommand{\D}{\mathcal D}
\newcommand{\M}{\mathcal M}
\newcommand{\CC}{Carnot-Carath\'eodory }

\newcommand{\A}{\mathcal A}
\newcommand{\B}{\mathcal B}
\newcommand{\sL}{\mathcal L}
\newcommand{\sLi}{\mathcal L_{\infty}}
\newcommand{\supp}{\operatorname{supp}}

\newcommand{\G}{\Gamma}
\newcommand{\x}{\xi}

\newcommand{\eps}{\epsilon}
\newcommand{\al}{\alpha}
\newcommand{\be}{\beta}
\newcommand{\p}{\partial}  
\newcommand{\lig}{\mathfrak}

\def\dist{\mathop{\varrho}\nolimits}

\newcommand{\BCH}{\operatorname{BCH}\nolimits}
\newcommand{\Lip}{\operatorname{Lip}\nolimits}
\newcommand{\Hol}{C}                             
\newcommand{\lip}{\operatorname{lip}\nolimits}
\newcommand{\capQ}{\operatorname{Cap}\nolimits_Q}
\newcommand{\pCap}{\operatorname{Cap}\nolimits_p}
\newcommand{\Om}{\Omega}
\newcommand{\om}{\omega}
\newcommand{\half}{\frac{1}{2}}
\newcommand{\e}{\epsilon}
\newcommand{\vn}{\vec{n}}
\newcommand{\X}{\Xi}
\newcommand{\tLip}{\tilde  Lip}

\newcommand{\Span}{\operatorname{span}}

\newcommand{\ad}{\operatorname{ad}}
\newcommand{\Hm}{\mathbb H^m}
\newcommand{\Hn}{\mathbb H^n}
\newcommand{\Hone}{\mathbb H^1}
\newcommand{\Lie}{\mathfrak}
\newcommand{\bg}{\Lie g}
\newcommand{\bG}{G}
\newcommand{\degree}{\text{deg}}
\newcommand{\Mn}{M^n}
\newcommand{\Rn}{\R^n}

\newcommand{\Layer}{V}
\newcommand{\hgrad}{\nabla_{\!H}}
\newcommand{\im}{\textbf{i}}
\newcommand{\nz}{\nabla_0}
\newcommand{\s}{\sigma}
\newcommand{\se}{\sigma_\e}

\newcommand{\ued}{u^{\e,\delta}}
\newcommand{\ueds}{u^{\e,\delta,\sigma}}
\newcommand{\tnabla}{\tilde{\nabla}}

\newcommand{\bx}{\bar x}
\newcommand{\by}{\bar y}
\newcommand{\bt}{\bar t}
\newcommand{\bs}{\bar s}
\newcommand{\bz}{\bar z}
\newcommand{\btau}{\bar \tau}
\newcommand{\bY}{\bar Y^{\e}}
\newcommand{\bd}{\bar{d}}

\newcommand{\LC}{\mbox{\boldmath $\nabla$}}
\newcommand{\Ne}{\mbox{\boldmath $n^\e$}}
\newcommand{\nuo}{\mbox{\boldmath $n^0$}}
\newcommand{\nuu}{\mbox{\boldmath $n^1$}}
\newcommand{\nue}{\mbox{\boldmath $n^\e$}}
\newcommand{\nuek}{\mbox{\boldmath $n^{\e_k}$}}
\newcommand{\dse}{\nabla^{H\Su, \e}}
\newcommand{\dso}{\nabla^{H\Su, 0}}
\newcommand{\tX}{\tilde X}

\newcommand{\Xie}{X^\epsilon_i}
\newcommand{\Xje}{X^\epsilon_j}
\newcommand{\Su}{\mathcal S}
\newcommand{\F}{\mathcal F}
\def\data{\X,\allowbreak C_D,\allowbreak C_P,\allowbreak \A_0,\allowbreak \A_1,\allowbreak p}

\title[Regularity for subelliptic PDE]{Regularity for subelliptic PDE \\through uniform estimates in multi-scale geometries}

\author[Capogna]{Luca Capogna}
\thanks{L. C. is partially supported by NSF award  DMS 1101478}
\address{Luca Capogna\\Department of Mathematical Sciences, Worcester Polytechnic Institute\\Worcester, MA 01609
}
\email{lcapogna@wpi.edu}

\author[Citti]{Giovanna Citti}
\address{Giovanna Citti\\ Dipartimento di Matematica, Universita di Bologna\\
Bologna, Italy}
\email{citti@unibo.it}

\begin{abstract} We aim at  reviewing and extending a number of recent results addressing stability of certain geometric and analytic estimates in the Riemannian approximation of subRiemannian structures. In particular we extend the recent work of the the authors with Rea \cite{CCR} and Manfredini \cite{MR3108875} concerning stability of doubling properties, Poincare' inequalities, Gaussian estimates on heat kernels and Schauder estimates from the Carnot group setting to the general case of H\"ormander vector fields. \end{abstract}

\maketitle

\tableofcontents
\setcounter{tocdepth}{3}
\section{Introduction}

A subRiemannian manifold as a triplet  $(M,\Delta, g_0)$ where $M$ is a connected, smooth manifold of dimension $n\in \N$, $\Delta$ denotes a  subbundle of $TM$  bracket-generating $TM$, and $g_0$ is a positive definite smooth, bilinear form  on $\Delta$,
see for instance \cite{Montgomery}.
Similarly to  the Riemannian setting, one  endows $(M,\Delta, g_0)$ with a metric space structure by defining the {\it Carnot-Caratheodory}  (CC) control distance: For any pair $x,y\in M$ set  \begin{multline} d_0(x,y)=\inf\{\delta>0 \text{ such that there exists a curve } \gamma\in C^\infty([0,1]; M) \text{ with endpoints }x,y \\\text{ such that }\dot\gamma\in \Delta (\gamma) \text{ and }|\dot\gamma|_g\le \delta\}.\notag\end{multline}
Curves whose velocity vector lies in $\Delta$ are called {\it horizontal}, their length  is defined in an obvious way. 
Subriemannian metrics
can be defined, by prescribing a smooth distributions of  vector fields $X=(X_1,\ldots ,X_m)$ in $R^n$, orthonormal with respect to $g_0$, and
 satisfying the H\"ormander finite rank condition
\begin{equation}    \label{Hor} rank \,\, Lie
(X_1,\ldots ,X_m)(x)=n, \quad \forall x\in \Omega.\end{equation}

\bigskip

When attempting to extend known Riemannian results to the subRiemannian setting one naturally is led to approximating the 
sub-Riemannian metric (and the associated  distance function $d_0(\cdot, \cdot) $) with
a one-parameter family of degenerating Riemannian metric (associated to distance functions $d_\e(\cdot, \cdot)$), which converge in the Gromov-Hausdorff sense as $\e\to 0$ to the original one.
This approximation is described in detail in from the point of view of the distance functions  in Section \ref{DEPS} and from the point of view of the 
Riemannian setting in Definition \ref{metrica-epsilon}.   The approximating distance functions $d_\e$  can be defined in  terms of an extended  generating frame of smooth vector fields
$X_1^\e,...,X_p^\e$, with $p\ge n$ and $X_i^\e=X_i$ for $i=1,...,m$,  that converges/collapses uniformly on compact sets to the original family$X_1,...,X_m$ as $\e\to 0$. This frame includes all the higher order commutators needed to bracket generate the tangent space. When coupled with uniform estimates, this method provides a strategy to extend known Riemannian results to the subRiemannian setting. Such approximations have been widely used since the mid-80's in a variety of contexts. As example we recall the work of  Debiard \cite{de:1981}, Koranyi \cite{kor:1983,kor:heisenberg}, Ge \cite{Ge}, Rumin \cite{MR1771424} as well as the references in \cite{Montgomery:book} and \cite{monti-tesi}.   More recently this technique has been used in the study of minimal surfaces and mean curvature flow in the Heisenberg group Starting from the existence theorem of Pauls \cite{Pauls:minimal},  and Cheng, Hwang and Yang \cite{chy}, to the regularity results by Manfredini and the authors \cite{ccm1}, \cite{CCM2}.
 Our work is largely inspired to the results  of Manfredini and one of us  \cite{cittimanfredini:uniform} where the  Nagel, Stein and Wainger estimates for the fundamental solution of subLaplacians  have been extended to the Riemannian approximants uniformly as $\e\to 0$.  In the following we list in more detail the nature  of the  stability estimates we investigate. Given a Riemannian manifold $(M^n,g)$, with a Riemannian smooth volume form  expressed in local coordinates $(x_1,...,x_n)$ as $d\ vol= \sqrt{g} dx_1...dx_n$, one can consider the corresponding heat operator acting on functions $u:M\to \R$, 
$$L_gu = \p_t u - \frac{1}{\sqrt{g}} \sum_{i,j=1}^n \p_i ( \sqrt{g} g^{ij} \p_j u).$$
The study of such operators is closely related to certain geometric and analytic estimates, namely: For $K\subset \subset M$ and $r_0>0$ there exists positive constants $C_D,C_P,..$ below depending on $K,r_0,g$ such that for all $x\in K$ and $0<r<r_0$, one has
\begin{itemize}
\item (Doubling property) \begin{equation}\label{D} vol (B(x,r)) \ge C_D vol (B(x,2r));\end{equation}
\item (Poincar\'e inequality)  $\int_{B(x,r)} |u-u_{B(x,r)}| dvol \le C_P r \int_{B(x,2r)} |\nabla_g u| d vol$;
\item (Gaussian estimates) If $h_g$ denotes the heat kernel of $L_g$, $x,y\in M$  and $t>0$ one has
\begin{multline}\label{gauss-bound-riem}C_g^{-1}  (vol(B(x,\sqrt{t})))^{-n/2} \exp(A_g \frac{ d(x,y)^2 }{t}) \\ 
\le |h(x,y,t) 
\le C_g (vol(B(x,\sqrt{t})))^{-n/2} \exp(B_g \frac{ d(x,y)^2 }{t})  \end{multline}
and if appropriate curvature conditions hold
\begin{equation}\label{grad-est-riem}|\p_t^s \p_{i_1} ... \p_{i_k}  h(x,y,t,s)\le C_{s,k,g}  t^{-s-\frac{k}{2}} (vol(B(x,t-s)))^{-n/2} \exp(B_G \frac{ d(x,y)^2 }{t-s});\end{equation}
\item (Parabolic  Harnack inequality) If $L_gu=0$ in $Q=M\times (0,T)$ and $u\ge 0$ then 
\begin{equation}\label{harnack-riem}\sup_{B(x,r)\times (t-r^2, t-r^2/2)}   u \le C_g \inf_{B(x,r)\times (t+r^2/2, t+r^2)}   u.\end{equation}

 \end{itemize}
 The connections between such estimates was made evident in the work of Saloff-Coste \cite{SC}  and Grigoryan \cite{grig}, who independently established the equivalence
 
 \medskip
 
 \centerline{ \it (Poincare)+(Doubling) $<=>$ Gaussian estimates \eqref{gauss-bound-riem}  $<=>$ Parabolic Harnack inequality \eqref{harnack-riem}.}
See also related works by Biroli and Mosco \cite{birolimosco}, and Sturm \cite{MR1387522}.

\bigskip

{\it This paper aims at describing the behavior of such estimates along a sequence of metrics $g_\e$, that  collapse to a subRiemannian structure as $\e\to 0$. We will prove that the estimates are stable as $\e\to 0$ and explore some of the consequences of this stability.} Although, thanks to the work of Jerison \cite{jer:poincare}, Nagel, Stein and Wainger \cite{NSW} and Jerison and Sanchez-Calle \cite{MR865430},  the Poincar\`e inequality, the doubling property and the Gaussian bounds are well known for subRiemannian structures, it is not immediate that they continue to hold uniformly in the approximation as $\e\to 0$. For one thing, the Riemannian curvature tensor is unbounded as $\e\to 0$, thus preventing the use of Li-Yau's estimates. Moreover, as $\e\to 0$ the Hausdorff dimension of the metric spaces $(M,d_\e)$, where $d_\e$ denotes the distance function associated to $g_\e$, typical does not remain constant  and in fact increases at $\e=0$ to the homogeneous dimension associated to the subRimannian structure. The term {\it multiscale} from the title reflects the fact that the blow up of the metric as $\e\to 0$ is Riemannian at scales less than $\e$ and subRiemannian at larger scales.

To illustrate our work we introduce a prototype for the class of spaces we investigate,   we consider the manifold $M=\R^2\times S^1$, with coordinates $(x_1,x_2,\theta)$. The horizontal distribution is given by 
$$\Delta=span\{ X_1,X_2\}, \text{ with }X_1=\cos\theta\p_{x_1}+\sin\theta \p_{x_2}, \text{ and } X_2=\p_\theta.$$ The subRiemannian metric $g_0$ is defined so that $X_1$ and $X_2$ form a orthonormal basis. This is the group of Euclidean isometries defined below in example \ref{heisenberg-ex}. For each $\e>0$ we also consider the Riemannian metric $g_\e$ on $M$ uniquely defined by the requirement that $X_1,X_2, \e X_3$ is an orthonormal basis, with $X_3=-\sin\theta \p_{x_1}+\cos\theta \p_{x_2}$. Denote by $d_\e$ the corresponding Riemannian distance, by $X_i^*$ the adjoint of $X_i$ with respect to Lebesgue measure and by $\Gamma_\e$ the fundamental solution of the Laplace-Beltrami operator $L_\e=\sum_{i=1}^3 X_i^* X_i$. Since $L_\e$ is uniformly elliptic, then there exists $C_\e,R_\e>0$ such that for $d_\e(x,y)<R_\e$ the fundamental solution will satisfy $$C_\e^{-1} d_\e(x,y)^{-1} \le \Gamma_\e(x,y) \le C_e d_\e(x,y)^{-1}.$$ As $\e\to 0$ this estimate will degenerate in the following way: $R_\e\to 0$, $C_\e\to \infty$ and for $\e=0$ one will eventually have
$$ \Gamma_0(x,y)\approx d_0(x,y)^{-2}.$$
As a result of the work in \cite{NSW} one has that  for each $\e>0$ there exists $C_\e>0$ such that
$$ C^{-1} \frac{d_\e^2(x,y)}{|B_\e(x,d(x,y))|} \le \Gamma_\e(x,y) \le C \frac{d_\e^2(x,y)}{|B_\e(x,d(x,y))|}.$$  The main result of 
\cite{cittimanfredini:uniform} was  to  provide stable bounds for the fundamental solution by proving that one can choose $C_\e$ independent of $\e$ as $\e\to 0$. In this paper we extend such stable bounds to the degenerate parabolic setting and to the more general subRiemannian setting.

\bigskip

Since our results will be local in nature, unless explicitly stated we will always assume that $M=\R^n$ and use as volume the Lebesgue measure.  The first result we present is due to Rea and the authors \cite{CCR} and concerns stability of the doubling property.

\begin{thrm}\label{Main-1}

 For every  $\e_0>0$, and  $K\subset \subset \R^n$ there exist constants $R,C>0$ depending on $K, \e_0$ and on the subRiemannian structure,  such that for every $\e\in [0,\e_0]$, $x\in K$ and  $0<r<R$, 
 $$|B_\e(x,2r)| \le C |B_\e(x,r)|.$$
 Here we have denoted by $B_\e$ the balls related to the $d_\e$ distance function.
 \end{thrm}
We present here a rather detailed proof of this result, amending some minor gaps in the exposition in \cite{CCR}.
If the subRiemannian structure is equiregular, as an original contribution of this paper, in Theorem \ref{epsilon-equiv} we also present a  quantitative version of this result, by introducing an explicit quasi-norms equivalent to $d_\e$. These families of  quasi-norms play a role analogue to the one played by the Koranyi Gauge quasi-norm  \eqref{gauge} in the Heisenberg group.
We also sketch the proof of   the stability of Jerison's Poincare inequality from \cite{CCR}.

\begin{thrm}\label{Poincare-epsilon}
 Let $K\subset \subset \R^n$ and $\e_0>0$. 
The vector fields $(X^\e_i)_{i=1\cdots p}$ satisfy the Poincare inequality
$$\int_{B_\e(x,R)} |u-u_{B_\e(x,r)}| dx \le C_P\int_{B_\e(x,2r)} |\nabla^\e u| dx$$
with a constant $C_P$ depending on $K, \e_0$ and the subRiemannian structure, but independent of $\e$. Here we have denoted by $\nabla^\e u$ the gradient of $u$ along the frame $X_1^\e,...,X_p^\e$.
\end{thrm}

Our next results concerns the stability, as $\e\to 0$, of the Gaussian estimates for the heat kernels associated to the family of 
second order,  sub-elliptic differential equations in non divergence form 
$$L_{\e, A} u\equiv \p_t u- \sum_{i,j=1}^p  a^\e_{ij} X_i^\e X_j^\e u=0 ,$$
in a cylinder $ Q=\Om\times (0,T)$. Here $\{a_{ij}^\e\}_{i,j=1,...,p}$ is a constant real matrix such that \begin{equation}\label{coerc1} \frac{1}{2}\Lambda^{-1} \sum_{i=1}^p \xi_i^2 
 \leq \sum_{i,j=1}^p a_{ij}^\e \xi_i \xi_j \leq 2\Lambda \sum_{i=1}^p \xi_i^2,\end{equation}
for all $\xi \in \R^p$,  uniformly in $\e>0$ and
 \begin{equation}\label{coerc2} \Lambda^{-1} \sum_{i=1}^m \xi_i^2 
 \leq \sum_{i,j=1}^m a_{ij}^\e \xi_i \xi_j \leq \Lambda \sum_{i=1}^m \xi_i^2,\end{equation}
 for all $\xi\in \R^m$ and $\e>0$.

\begin{thrm}  Let $K\subset \subset \R^n, \Lambda>0$ and $\e_0>0$. 
The fundamental solution $\Gamma_{\e, A}$ of the operator 
 $L_{\e, A}$, is a kernel with  exponential 
decay of order $2$, uniform with respect to $\epsilon\in [0,\e_0]$ 
and for any coefficients matrix  $A$ satisfying the bounds above for  the  fixed $\Lambda>0$. In particular, the following estimates hold:

\begin{itemize}
\item{For every $K\subset\subset\Omega$ 
there exists a constant $C_\Lambda>0$  depending on $\Lambda$ but independent of $\e\in [0,\e_0]$, and of the matrix 
$A$ such that for each $\e\in [0,\e_0]$, $x,y\in K$ and $t>0$ one has
\begin{equation}\label{gaussiana}
C_\Lambda^{-1} \frac{e^{-C_\Lambda\frac{d_\e(x,y)^2}{t}}} {|B_\e (x, \sqrt{t})|}\le P_{\e, A^\e}(x,y,t)\le C_\Lambda\frac{ e^{-\frac{d_\e(x,y)^2}{C_\Lambda t}}} {|B_\e (x, \sqrt{t})|}.
\end{equation}}
\item{For  $s\in \N$ and $k-$tuple $(i_1,\ldots,i_k)\in \{1,\ldots,m\}^k$ there exists a constant $C_{s,k}>0$ depending only on $k,s,X_1,...,X_m,\Lambda$ such that
\begin{equation}
|(\p_t^s X_{i_1}\cdots X_{i_k} P_{\e, A^\e})(x,y,t)| \le C_{s,k}    \frac{t^\frac{-2s-k}{2} e^{-\frac{d_\e(x,y)^2}{C_\Lambda t}}} {|B_\e(x, \sqrt{t})|}
\end{equation}
for all $x,y\in K$ and $t>0$.}
\item{For any $A_1,A_2\in M_\Lambda$, $s\in \N$ and $k-$tuple $(i_1,\ldots,i_k)\in \{1,\ldots,m\}^k$ there exists $C_{s,k}>0$ depending only on $k,s,X_1,...,X_m, \Lambda$ such that
\begin{equation}\label{AmenA}
|(\p_t^s X_{i_1}\cdots X_{i_k} P_{\e, A_1})(x,y,t) - \p_t^s X_{i_1}\cdots X_{i_k} P_{\e, A_2})(x,y, t) |\leq
\end{equation}
$$ \le ||A_1 - A_2||C_{s,k}  \frac{t^\frac{h-2s-k}{2} e^{-\frac{d_\e(x,y)^2}{C_\Lambda t}}} {|B_\e(x, \sqrt{t})|},$$ where $||A||^2:=\sum_{i,j=1}^n a_{ij}^2$.}
\end{itemize}

Moreover, if $\Gamma_A$ denotes the fundamental solution of the operator 
$L_\A=\sum_{i,j=1}^m a_{ij}^0X_iX_j$, then one has
\begin{equation}{ X}^\e_{i_1}\cdots { X}^\e_{i_k} \p_t^s  \Gamma_{\e, A^\e}\to {X}_{i_1}\cdots {X}_{i_k}\p_t^s \Gamma_{A^0}\end{equation}
as $\e\to 0$  uniformly on compact sets and in a dominated way  on subcompacts of $\Om$.

\end{thrm}

One of our main result in this paper is the extension  to the H\"ormander vector fields setting 
 of the Carnot groups Schauder estimates established in previous work with Manfredini  in \cite{CCM3}.
To prove such extension we combine the Gaussian bounds above  with a refined  version of Rothschild and Stein \cite{Roth:Stein} freezing and lifting scheme, adapted to the multi-scale setting, to establish Schauder type  estimates which are uniform in $\e\in [0,\e_0]$,  for  the family of 
second order,  sub-elliptic differential equations in non divergence form 
$$L_{\e, A} u\equiv \p_t u- \sum_{i,j=1}^n  a^\e_{ij}(x,t) X_i^\e X_j^\e u=0 ,$$
in a cylinder $ Q=\Om\times (0,T)$.  Our standing assumption is that the coefficients of the operator satisfy \eqref{coerc1}, and \eqref{coerc2} for some fixed $\Lambda>0$.

%
%
%
%

\begin{thrm}\label{main-schauder}
Let   $\alpha\in (0,1), f\in C^{\infty}(Q)$ and $w$ be a smooth  solution of $L_{\e, A}w=f$ on ${Q}$.
Let $K$ be a compact sets such that  $K\subset\subset {Q}$,  set $2\delta=d_0(K, \p_p Q)$ and
denote by $K_\delta$ the $\delta-$tubular neighborhood of $K$. Assume that 
there exists a constant $C>0$ such that 
$$ || a_{ij}^\e||_{C^{k,\alpha}_{\e,X}(K_\delta)} \leq C,$$ for some value $k\in \N$ and for every $\e\in [0,\e_0]$.
There exists a constant $C_1>0$ depending  on
$\alpha$, $C$, $\e_0$, $\delta$, and the constants in Proposition \ref{uniform heat kernel estimates},
but independent of $\e$,  such that
$$||w||_{C^{k+2, \alpha}_{\e,X}(K)} \leq C_1 \left( ||f||_{C^{k,\alpha}_{\e,X}(K_\delta)}+ ||w||_{C^{k+1, \alpha}_{\e,X}(K_\delta)}\right). $$
Here we have set 
 $$||u||_{C_{\e,X}^{\alpha}({Q})}=\sup_{(x,t)\neq(x_{0},t_0)} \frac {|u(x,t) - u(x_{0},t_0)|}{\tilde d_{\e }^\alpha((x,t), (x_{0},t_0))}+ \sup_{Q} |u|.$$
 and 
 if $k\geq 1$  we have  let  $u \in
C_{\e,X}^{k,\alpha}({Q})$  if for all $i=1,\ldots ,m$, one has 
 $X_i \in C_{\e,X}^{k-1,\alpha}({Q})$.
\end{thrm}

Analogous estimates in the $L^p$ spaces, for operators independent of $\e$ are well known (see for instance \cite{Roth:Stein} for the constant coefficient case and \cite{BLU} for the Carnot group setting). 
Our result yield  a stable version, as $\e\to 0$, of such estimates, which is valid for any family of H\"ormander vector fields.

\begin{thrm}
Let   $\alpha\in (0,1), f\in C^{\infty}(Q)$ and $w$ be a smooth  solution of $L_{\e, A}w=f$ on ${Q}$. Let $K$ be a compact sets such that  $K\subset\subset {Q}$,  set $2\delta=d_0(K, \p_p Q)$ and
denote by $K_\delta$ the $\delta-$tubular neighborhood of $K$.   Assume that 
there exists a constant $C>0$ such that 
$$ || a_{ij}^\e||_{C^{k,\alpha}_{\e,X}(K_\delta)} \leq C,$$ for some value $k\in \N$ and for every $\e\in [0,\e_0]$.
For any $p>1$, there exists a constant $C_1>0$ depending  on
$p, \alpha$, $C$, $\e_0$, $\delta$, and the constants in Proposition \ref{uniform heat kernel estimates},
but independent of $\e$,  such that
$$||w||_{W^{k+2, p}_{\e,X}(K)} \leq C_1 \left( ||f||_{W^{k,p}_{\e,X}(K_\delta)}+ ||w||_{W^{k+1, p}_{\e,X}(K_\delta)}\right). $$
Here we have set $$||w||_{W^{k,p}_{\e,X}} :=\sum_{i=1}^k \sum_{| I |=i} ||X^\e_{i_1} ...  X^\e_{i_k}  w||_{L^p}.$$
\end{thrm}

\section{Definitions and preliminary results}

\input{approx}
\input{stability}

\section{Stability of the  Poincar\'e
inequality}

In this section we will focus on the Poincar\'{e} inequality and prove that it holds
 with  a choice of a constant which is stable as  $\epsilon\to 0$.  Our argument rests on results of Lanconelli and Morbidelli \cite{LanMor} whose proof, in some respects, simplifies the method used by Jerison in \cite{jer:poincare}. Using some Jacobian estimates from
  \cite{GarofaloNhieu:lip1998} or  \cite{frss:embedding} we will establish that the assumptions required in the key result \cite[Theorem 2.1]{LanMor} are satisfied independently from $\e\ge 0$.  We start by recalling

\begin{thrm}\cite[Theorem 2.1]{LanMor} \label{LM}
 Assume that the doubling condition $(D)$ is satisfied and there exist a sphere $B_\epsilon(x_0, r)$, a cube $Q_\epsilon \subset \R^n$
and a map $E: B_\e(x_0, r)\times Q_\epsilon \rightarrow \R^n$ satisfying the following conditions:
\begin{itemize}
\item [i)]    $B_\e(x_0, 2 r)\subset E(x, Q_\epsilon)$ \quad for every $x\in B_\e(x_0, r)$
\item[ii)]the function $u \mapsto E(x, u)$ is one to one on the box $ Q_\epsilon$ as a function of the variable $u$ and there exists a constant  $\alpha_1>0$ such that $$\frac{1}{\alpha_1} |JE(x,0)| \leq |JE(x,u)| \leq \alpha_1 |JE(x,0)| \quad  \text{ for every  } u \in Q_\epsilon$$
\end{itemize}
Also assume that there exists a positive constant $\alpha_2$, and a function $\gamma:  B_\e(x_0, r) \times Q_\epsilon \times [0,\alpha_2 r]\rightarrow \R^n$
satisfying the following conditions
\begin{itemize}
 \item[iii)]  For every $(x,u) \in  B_\e(x_0, r) \times Q_\epsilon$ the function
$t \mapsto \gamma(x,u,t)$ is a subunit path connecting $x$ and $E(x,u)$
 \item[iv)] For every $(h,t) \in  B_\e(x_0, r) \times Q_\epsilon$ the function $x \mapsto \gamma(x,u,t)$ is a one-to-one map and there exists a constant $\alpha_3>0$ such that
     $$\inf_{ B_\e(x_0, r)\times Q_\epsilon} \Big|det \frac{\partial \gamma}{\partial x}\Big|\geq\alpha_3$$
\end{itemize}
Then there exists a constant $C_P$ depending only on the constants $\alpha_1, \alpha_2, \alpha_3$ and the doubling constant $C_D$ such that (P) is satisfied.
\end{thrm}

We are now ready to prove Theorem \ref{Poincare-epsilon}
\begin{proof}
All one needs to establish is  that the assumptions of Theorem \ref{LM} are satisfied unformly in $\e$ on a metric ball. Apply Proposition \ref{indipendenza} and Theorem \ref{MAINNSWeps}   with $K=B_\epsilon(x_0, r)$ and choose the constants $C_i$ produced by these results.
 Set $Q_\epsilon=Q_\e(\frac{3 C_{1}}{C_2} r)$ and let
$$E(x,u)= \Phi_{\epsilon, 0,x}(u), \text{ defined on }
K \times Q_\epsilon\rightarrow \R^n.$$

To establish assumption (i) of Theorem \ref{LM} it suffices to note that by virtue of  condition (iii) in Theorem  \ref{MAINNSWeps} one has that for $x\in B_\epsilon(x_0, r)$, $$B_\e(x_0, 2 r)\subset B_\e(x, 3 r)\subset E(x, Q_\epsilon).$$
 Assumption (ii) in Theorem \ref{LM} is a direct consequence of condition (ii) in Theorem  \ref{MAINNSWeps}, with $\alpha_1= 16$.  Chow's connectivity theorem implies that $E(x,u)$ satisfies assumption (iii), with a function $\gamma$, piecewise expressed as exponential mappings of vector fields of $\e-$degree one. Let us denote   $(X^\e_i)_{i\in I_\e}$ the required vector fields. With this choice of path, it is known (see for example \cite[Lemma 2.2]{GarofaloNhieu:lip1998} or  \cite[pp 99-101]{frss:embedding}) that $x\rightarrow \gamma(x,u,t)$ is a $C^1$ path, with Jacobian determinant
$$\bigg|det \frac{\partial \gamma}{\partial x}(x,u,t)\bigg|= 1 + \psi(x,u,t),$$
for a suitable function $\psi(x,u,t)$
satisfying $$|\psi(x,u,t)|\leq cr, \text{  on  }K\times Q_\e \times [0,cr].$$
Since the constant $c$  depends solely on the Lipschitz constant of the vector fields $(X^\e_i)_{i\in I_\e}$ then  it can be chosen independently of $\e$. As a consequence condition (iv) is satisfied and the proof is concluded.

\end{proof}

\section{Stability of Heat Kernel Estimates}

\input{HeatKernel}

\section{Stability of interior Schauder estimates} 
\input{Schauder}

\section{Application I: Harnack inequalities for degenerate parabolic quasilinear equations hold uniformly in $\e$}
\input{Harnack}

\section{Application II: Regularity for quasilinear sub elliptic PDE through Riemannian approximation}

As an illustration of the usefulness of the uniform estimates established above, in this section we want to briefly sketch the  strategy used in \cite{CCM4} and \cite{MR3108875}, where the Riemannian approximation scheme is used to establish regularity for the graph solutions of the Total Variation flow
\begin{equation}\label{pde01}\frac{\p u}{\p t} =\sum_{i=1}^m X_i  \Big(\frac{X_iu}{\sqrt{1+|\nabla_0 u|^2}}\Big),
\end{equation} 
and for the graphical solutions of the  mean curvature flow
\begin{equation}\label{pde02}\frac{\p u}{\p t} =\sqrt{1+|\nabla_0 u|^2}\sum_{i=1}^m X_i  \Big(\frac{X_iu}{\sqrt{1+|\nabla_0 u|^2}}\Big).
\end{equation} 
In both cases $\Om\subset G$ is a bounded open set, with $G$ is a Lie group, free up to step two, but not necessarily nilpotent.

We will consider solutions arising as limits of solutions of the analogue Riemannian flows, i.e.
\begin{equation}\label{pdee1}\frac{\p u_{\e}}{\p t} = h_{\e}=\sum_{i=1}^n X_i^{\e}\Big(\frac{X_i^{\e}u_{\e}}{W_{\e} }\Big)\quad \text{ for }x\in \Om, \; t>0, 
\end{equation}
and
\begin{equation}\label{pdee2}\frac{\p u_{\e}}{\p t} =W_\e h_{\e}=W_\e\sum_{i=1}^n X_i^{\e}\Big(\frac{X_i^{\e}u_{\e}}{W_{\e} }\Big)=\sum_{i,j=1}^na_{ij}^\e (\nabla_\e u_\e) X_i^\e X_j^\e u_\e\quad \text{ for }x\in \Om, \; t>0, 
\end{equation}
where, $h_\e$ is the mean curvature of the graph of $u_\e (\cdot, t)$ and 
\begin{equation}\label{defaij} W_\e^2=1+|\nabla_\e u_\e|^2=
1+\sum_{i=1}^n (X_i^\e u_\e)^2 \text{ and } a_{ij}^\e(\xi)=  \delta_{ij}-\frac{\xi_i \xi_j}{1+|\xi|^2} ,
\end{equation}
 for all $\xi\in \R^n$.

The main results in \cite{CCM4} and \cite{MR3108875} concern long time existence of solutions of the initial value problems
\begin{equation}\label{ivp}
 \Bigg\{
 \begin{array}{ll}
 \p_t u_\e= h_\e W_\e  &\text{ in }Q=\Om\times(0,T) \\
 u_\e=\varphi &\text{ on } \p_p Q,
 \end{array} \quad \text{ and } 
 \Bigg\{
 \begin{array}{ll}
 \p_t u_\e= h_\e   &\text{ in }Q=\Om\times(0,T) \\
 u_\e=\varphi &\text{ on } \p_p Q,
 \end{array} 
 \end{equation}
 with $\p_p Q=(\Om\times \{t=0\})\cup (\p\Om \times (0,T))$ denoting the parabolic boundary of $Q$.

\begin{thrm} \label{global in time  existence results}
Let $G$ be a Lie group of step two,
 $\Om\subset G$ a bounded, open, convex  set (in a sense to be defined later)  and $\varphi\in C^2(\bar{\Om})$.  There exists  unique solutions
$u_\e \in C^{\infty}(\Om\times (0,\infty))\cap L^\infty((0,\infty),C^1(\bar \Om))$ of the two  initial value problems in
\eqref{ivp}, and  for each $k\in \N$ and $K\subset \subset Q$, there exists $C_k=C_k(G,\varphi,k,K,\Om)>0$ not depending on $\e$ such that
\begin{equation}\label{stable estimates}
||u_\e||_{C^k(K)} \le C_k.
\end{equation}
\end{thrm}

\begin{cor} 
Under the assumptions of  Theorem \ref{global in time  existence results},
as $\e\to 0$ 
the solutions $u_\e$ of either flow converge uniformly (with all theirs derivatives) on compact subsets of $Q$ to the unique,  smooth solution 
 of the corresponding sub-Riemannian  flow  in $\Om\times (0,\infty)$ with initial data $\varphi$.
\end{cor}

The proof of this result rests crucially on the estimates established in this paper. In the following we list the main steps. First of all we note that in view of the short time existence result in the Riemannian setting we can assume that locally $u_\e$ are smooth both in time and space.

\begin{enumerate}
\item {\bf Interior gradient bounds.}  Denote by  right $X_i^r$ the left invariant frame corresponding to $X_i's$ and observe that these two frames commute. For both flows, consider solutions $u_\e\in C^3(Q)$ and denote
 $v_0=\partial_t u_\e$, $v_i = X_i^{r}u_\e$ for $i=i, \ldots, n$. Then
 for every $h=0,\ldots ,n$ one has that $v_h$ is a solution of
\begin{equation}\label{diff-eq}\p_t v_h= X_i^\e ( a_{ij} X_jv_h )= a_{ij}^\e(\nabla_\e u_\e) X_i^\e X_j^\e v_h + \p_{\xi_k} a_{ij}^\e(\nabla_\e u)X_i^\e X_j^\e u_\e X_k^\e v_h,\end{equation}
where
$$ a_{ij}^\e(\xi) = \frac{1}{\sqrt{1+|\xi|^2}}\Big(\delta_{ij}- \frac{\xi_i \xi_j}{1+|\xi|^2}\Big),$$ for the total variation flow, while
$$ a_{ij}^\e(\xi) =  \delta_{ij}- \frac{\xi_i \xi_j}{1+|\xi|^2},$$
for the mean curvature flow. The weak parabolic maximum principle yields  that there exists $C=C(G,||\varphi||_{C^2( \Om )})>0$
such that for every compact subset $K\subset \subset \Om$ one has
$$\sup_{K \times [0,T)} |\nabla_1 u_\e|  \leq \sup_{\partial_p Q}(|\nabla_1 u_\e| + |\partial _ t u_\e|),$$
where $\nabla_1$ is the  full $g_1-$Riemannian gradient. This yields the desired unform  interior gradient bounds. This argument works in any Lie group, with no restrictions on the step.

\item {\bf Global gradient bounds.} The proof of the boundary gradient estimates is more delicate and depends crucially on the geometry of the space. In particular the argument we outline here only holds in step two groups $G$ and for domains $\Om\subset G$ that are locally Euclidean convex when expressed in the Rothschild-Stein preferred coordinates introduced in \eqref{phi0}.
In \cite{CCM4} we use the Rothschild-Stein osculation Theorem \ref{rs-5} to construct a rather explicit barrier function at any boundary point and then to
 conclude we apply the comparison principle \cite[Theorem 3.3]{CC}. This argument also shows that the solutions $v_h$ to \eqref{diff-eq} are bounded.
\item {\bf Harnack inequalities and $C^{1,\alpha}$ estimates.} We have proved in Theorem \ref{Main-1} and Theorem \ref{Poincare-epsilon}, that $(G,d_\e)$ is a $2-$admissible geometry in the sense of Definition \ref{admissible}, with Doubling and Poincare constants uniform in $\e\ge 0$. As a consequence we can apply the Harnack inequalities in Theorem \ref{main-th} and Proposition \ref{weak-harnack} to the bounded solutions $v_h$ of \eqref{diff-eq}, thus yielding the $C^{1,\alpha}$ uniform interior estimates.
\item {\bf Schauder estimates and higher order estimates}  The uniform Gaussian estimates and  Schauder estimates in Theorem \ref{main-schauder} applied to \eqref{diff-eq} yield the higher order estimates and conclude the proof.

\end{enumerate}

\bibliographystyle{acm}
\bibliography{survey-cc}
\end{document}

%% file: approx.tex
Let $X=(X_1,...,X_m)$ denote a collection of  smooth vector fields defined in an open subset $\Om\subset\R^n$ satisfying H\"ormander's finite rank condition \eqref{Hor}, that is {\it there exists an integer $s$ such that the set of all vector fields, along with their commutators up to order $s$ spans $\R^n$ for every point in $\Omega$},\begin{equation}    rank \,\, Lie
(X_1,\ldots ,X_m)(x)=n, \quad \text{ for all }l x\in \Omega.\end{equation}

\begin{ex}\label{heisenberg-ex} The standard example for such families is the Heisenberg group $\Hone$. This is a Lie group whose underlying manifold is $\R^3$ and is endowed with a group law $(x_1,x_2,x_3)(y_1,y_2,y_3)=(x_1+y_1, x_2+y_2, x_3+y_3-(x_2y_1-x_1y_2))$. With respect to such law one has that the vector fields $X_1=\p_{x_1}-x_2 \p_{x_3}$ and $X_2=\p_{x_2}+x_1\p_{x_3}$ are left-invariant. Together with their commutator $[X_1,X_2]=2\p_{x_3}$ they yield a basis of $\R^3$. A second example is given by the classical group of rigid motions of the plane, also known as the {\it roto-translation} group $\mathcal {RT}$.  This is a Lie group with underlying manifold $\R^2\times S^1$
and a group law $(x_1,x_2,\theta_1)(y_1,y_2,\theta_2)=(x_1+y_1\cos \theta-y_2\sin\theta, x_2+y_1\sin \theta+y_2\cos\theta, \theta_1+\theta_2)$.
\end{ex}

 Following Nagel, Stein and Wainger, \cite[page 104]{NSW} we define 
\begin{equation}\label{NSWspaces} X^{(1)}=\{ X_1,...,X_m\}, \ X^{(2)}=\{ [X_1,X_2],...,[X_{m-1},X_m]\}, \ etc. ...\end{equation}
letting $X^{(k)}$ denote the set of all commutators of order $k=1,...,r$.
Indicate by $Y_1,...,Y_p$ an enumeration of the components of $X^{(1)}, X^{(2)},...,X^{(r)}$ such that $Y_i=X_i$ for every $i\leq m$.  If $Y_k\in X^{(i)}$ we say that  $Y_k$ has a {\it formal} degree  $d(Y_k)=d(k) =i$. The collection of vector fields $\{Y_1,...,Y_p\}$ spans $\R^n$ at every point.

\begin{ex} If we consider the Heisenberg group vector fields $X_1=\p_{x_1}-x_2 \p_{x_3}$ and $X_2=\p_{x_2}+x_1\p_{x_3}$ with  $(x_1,x_2,x_3)\in \R^3$, then $X^{(1)}:=\{X_1,X_2\}$ and $X^{(2)}=\{2\p_{x_3}\}$. If we instead consider the vectors arising from the group of roto-translations one has $X_1=\cos\theta \p_{x_1}+\sin \theta \p_{x_2}$ and $X_2=\p_{\theta}$ with $(x_1,x_2, \theta)\in \R^2\times S^1$ and 
$X^{(1)}=\{X_1, X_2\} $ and $X^{(2)}=\{ \sin \theta \p_{x_1} - \cos \theta \p_{x_2}\}$. 
\end{ex}
\begin{ex}\label{esempio32} Note that  the sets $X^{(i)}$ may have  non-trivial intersection. For instance, consider the vector fields 
$$X_1=\cos\theta \p_{x_1}+\sin \theta \p_{x_2}; \ X_2=\p_{\theta}; \ X_3=\p_{x_3}; \text{ and }X_4=x_3^2 \p_{x_4}$$
in $(x_1,x_2,x_3,x_4,\theta)\in \R^4\times S^1$.  In this case $r=3$ and 
$$X^{(1)}=\{X_1, X_2, X_3, X_4\}; \ X^{(2)}=\{ \sin \theta \p_{x_1} - \cos \theta \p_{x_2}, \  2x_3\p_{x_4}  \}; \text{ and }X^{(3)}= \{\pm X_1, 2\p_{x_4}\}$$
with $Y_1=X_1, ..., Y_4= X_4, Y_5=\sin \theta \p_{x_1}-\cos \theta \p_{x_2}, y_6= 2x_3 \p_{x_4}, Y_7=X_1, Y_8=-X_1,$ and $Y_{10}=2\p_{x_4}$.
\end{ex}
\subsection{Carnot-Caratheodory distance} 
For each $x, y\in \Om$ and $\delta>0$ denote by $\Gamma(\delta)$ the space of all absolutely continuous curves $\gamma:[0,1]\to \R^n$, joining $x$ to $y$ (i.e., $\gamma(0)=x$ and $\gamma(1)=y$) which are tangent a.e. to the horizontal distribution $span \{X_1,...,X_m\}$, and such that if we write $$\gamma'(t)=\sum_{i=1}^m \alpha_i(t) X_i|_{\gamma(t)},$$ then
$\sum_{i=1}^m |\alpha_i(t)|\le \delta$ a.e. $t\in [0,1]$. The Carnot-Caratheodory distance between $x$ and $y$ is defined to be
\begin{equation}\label{distance-CC}
d_0(x,y):= \inf_{\Gamma(\delta)\neq \{\} } \delta.
\end{equation}
 In \cite{NSW}, the authors introduce several other  distances that eventually are proved to be equivalent to $d_0(x,y)$. The equivalence itself yields new insight into the Carnot-Caratheodory distance. Because of this,  we will remind the reader of one of these distances. For each $x,y\in \Om$ and $\delta>0$ denote by $\hat\Gamma(\delta)$ the space of all absolutely continuous curves $\gamma:[0,1]\to \R^n$, joining $x$ to $y$ and such that if one writes  $$\gamma'(t)=\sum_{i=1}^p\beta_i(t) Y_i|_{\gamma(t)},$$ then $|\beta_i (t)| \le \delta^{d(i)}.$ One then sets 
 $$\hat d (x,y):= \inf_{\hat \Gamma(\delta)\neq \{\}} \delta.$$
It is fairly straightforward (see \cite[Proposition 1.1]{NSW} to see that 
\begin{prop}\label{p11}
The function $\hat d$ is a distance function in $\Om$ and for any $K\subset \subset \Om$ there exists 
$C=C(X_1,...,X_m, K)>0$ such that
$$C^{-1} |x-y|\le \hat d(x,y)\le C |x-y|^{\max_i d(i)}.$$
\end{prop}
It is far less trivial to prove the following (see \cite[Theorem 4]{NSW})
\begin{thrm} The distance functions $d_0$ and $\hat d$ are equivalent.
\end{thrm}

\subsection{The approximating distances} \label{DEPS}

There are several possibile definitions for Riemannian distance functions which approximate a Carnot-Caratheodory metric in the Gromov-Hausdroff sense. 

\begin{dfn}\label{prima-def}
Let $\{Y_1,...,Y_p\}$ be a generating family of vector fields constructed as in \eqref{NSWspaces} from a family of H\"ormander vector fields $X_1,...,X_m$. For every $\e>0$ denote by $d_\e(\cdot, \cdot)$   the Carnot-Caratheodory metric  associated to the family of vector fields $(X_1^\e, ...,X_{p}^\e),$ defined as

\begin{equation}\label{definefields}
X_i^\e=\Bigg\{ \begin {array}{lll} Y_i & \text{ if } i\leq m,  \\
\e^{d(i) -1 } Y_i &\text{ if } m+1 \leq i \leq p, \\ Y_{i-p+m} &\text{ if } p+1 \leq i \leq 2p -m
\end{array}.
\end{equation}
We will also define an extension of the degree function, setting $d_\e(i)=1$ for all $i\leq p$, and $d_\e(i) = d(i-p+m)$ if  $i\geq p+1$.  In order to simplify notations we will denote $X=
 X^0$,  $d_0=d$
 and use the same notation for both families of vector fields
 (dependent or independent of $\e$).
\end{dfn}
Note that for every $\e\in (0,\bar \e)$ the sets $\{X_i^\e\}$ extends the original family of vector fields $(X_i)$ to a  new families of vector fields satisfying assumption (I) on page 107 \cite{NSW}, i.e. there exist smooth functions $c_{jk}^l$, depending on $\e$, such that
$$[X^\e_j,X^\e_k] =\sum_{d_\e(l)\leq d_\e(j) + d_\e(k)} c_{jk}^l {X_l^\e}$$
and $$\{X^\e_j\}_{j=1}^{2p-m} \text{ span } \R^n  \text{ at every point }.$$

\begin{rmrk}
Note that the coefficients $c_{jk}^l$ will be unbounded as $\e\to 0$. In principle this could be a problem  as the
doubling constant in the proof in \cite{NSW} depends indirectly from the $C^r$ norm of these functions. In this survey we will describe a result, originally proved in  \cite{CCR}, showing  that this is not the case. \end{rmrk}

 \begin{rmrk} It follows immediately from the definition that  for fixed $x,y\in \Om$ the function $d_\e(x,y)$ is decreasing in $\e$ and for every $\e\in (0,\bar \e)$,
 $$d_0(x,y)\ge d_\e(x,y)$$ \end{rmrk}

\begin{rmrk}\label{subR} Let us consider a special case where $\dim\text{ span }(X_1,...,X_m)$ is constant and the vector fields $X_1,...,X_p$ are chosen to be linearly independent in $\Omega$. In this case we can consider  two positive defined symmetric quadratic forms $g_0,$ and $ \lambda$ defined respectively on the distribution $H(x)=\text{ span }(X_1,...,X_m)(x)$, for $x\in \Om$ and on $H^\perp(x)$. The product metric $g_0\oplus \lambda$ is then a Riemannian metric on all of $T \Omega$.  The form $g_0$ is called a {\it subRiemannian} metric on $\Om$, corresponding to $H$.
 Next, for every $\e\in (0,\bar \e]$  reconsider the rescaled metric $g_\e:=g_0\oplus \e^{-1}\lambda$ and the corresponding Riemannian distance function $d_\e$ in $\Omega$.  The latter is bi-Lipschitz equivalent to the distance $d_\e$ defined above.  In \cite[Theorem 1.1]{Ge} Ge proved that  that as metric spaces, the sequence $(\Om, d_\e)$ converges to $(\Om, d_0)$ as $\e\to 0$ in the sense of Gromov-Hausdorff. In this limit the Hausdorff dimension of the space degenerates from coinciding with the topological dimension, for $\e>0$, to a value $Q>n$ which may change from open set to open set. We will go more in detail about this point in the next section. In this sense the limiting approximation scheme we are using can be described by the collapsing of a family of Riemannian metric to a subRiemannian metric. See also \cite[Theorem 1.2.1]{monti-tesi} for yet another related Riemannian approximation scheme.
 \end{rmrk}

%

\begin{rmrk}From different perspectives, note that the subLaplacian associated to the family $X_1^\e,...,X_m^\e$ i.e. $\mathcal{L}u=\sum_{i=1}^m X_i^{\e,2}u$ is an elliptic operator for all $\e>0$, degenerating to a subelliptic operator for $\e=0$.
\end{rmrk}

\subsection{A special case: The Heisenberg group $\mathbb H^1$} In this section we describe the behavior of the distance $d_\e$ (and of the corresponding metric balls $B_\e(x,r)$ as $\e\to 0$, by looking at the special case of the Heisenber group. In this setting we will also provide an elementary argument showing that the doubling property holds uniformly as $\e\to 0$. 

Consider the vector fields from Example \ref{heisenberg-ex}     $X_1=\p_{x_1}-x_2 \p_{x_3},$  $X_2=\p_{x_2}+x_1\p_{x_3}$ and $X_3=\p_{x_3}$ with  $(x_1,x_2,x_3)\in \R^3$.
The  Carnot-Carath\'eodory distance $d_0$ associated to the subRiemannian metric defined by the orthonormal frame $X_1,X_2$ is equivalent
to a more explicitly defined pseudo-distance function, that we
 call  {\sl gauge distance}, defined as
\begin{equation}\label{gauge}
|x|^{4}=(x_1^2+x_2^2)^2+x_3^2, \text{ and
} \rho(x,y)=|y^{-1}x|,
\end{equation}
where $y^{-1}=(-y_1,-y_2,-y_3)$ and $y^{-1}x=(x_1-y_1,x_2-y_2, x_3-y_3- (y_1x_2-x_1y_2))$ is the {\it Heisenberg group multiplication}.
\begin{lemma} For each $x\in \R^3$,
\begin{equation}\label{stesso} A^{-1}|x|\le d_0(x,0)\le A|x|,\end{equation}
for some constant $A>0$. \end{lemma}
\begin{proof} Observe that the 1-parameter family of diffeomorpthisms $$(x_1,x_2,x_3)\to \delta_\lambda (x_1,x_2,x_3):=(\lambda x_1, \lambda x_2, \lambda^2 x_3)$$
satisfies 
$|\delta_\lambda (x)|=\lambda |x|$, and $d\delta_\lambda X_i=\lambda X_i\circ \delta_\lambda$ for $i=1,2$. Consequently $d_0(\delta_\lambda(x),\delta_\lambda(y))=\lambda d_0 (x,y)$, and $\delta_\lambda(B(0,1))=B(0,\lambda)$. Since the unit ball $B(0,1)$ is a bounded open neighborhood of the origin, it will contain a set of the form $|x|\le A^{-1}$ and will be contained in a set of the form $|x|\le A$. By applying $\delta_\lambda$ we then have that for any $R>0$, 
$$\{ x\in \R^3| |x|\le A^{-1} R\} \subset B(0,R)\subset \{ x\in \R^3| |x|\le A R\}$$
concluding the proof of \eqref{stesso}.
\end{proof}

\begin{rmrk} Since the Heisenberg group is a Lie group, then it is natural to use a left-invariant volume form to measure the size of sets, namely the Haar measure. It is not difficult to see \cite{CorwinGreenleaf} that the Haar measure coincides with the Lebesgue measure in $\R^3$. It follows immediately from the previous lemma that the corresponding volume of a ball $B(x,r)$ is \begin{equation}\label{volume0} |B(x,r)|=Cr^4.\end{equation}
As a consequence one can show that the Hausdorff dimension of the metric space $(\Hone, d_0)$ is $4$. The Hausdorff dimension of any horizontal curve (i.e. tangent to the distribution generated by $X_1$ and $X_2$) is $1$, while the Hausdorff dimension of the vertical $z$-axis is $2$.
\end{rmrk}

Next we turn our attention to the balls in the metrics $g_\e$ and the associated distance functions $d_\e$. To better describe the approximate shape of such balls we  define the pseudo-distance function $d_{G,\e} (x,y)=N_\e (y^{-1}x)$ corresponding to the  regularized gauge function
\begin{equation}\label{Nepsilon}
N_\e^2(x)=x_1^2+x_2^2+
 \min\bigg\{ |x_3|  , \e^{-2}x_3^2
 \bigg\}.
\end{equation}
Our next goal is to show that  the Riemannian distance function $d_\e$ is well approximated by the gauge  pseudo-distance $d_{G,\e}$. 
\begin{lemma} \label{N=d}
There exists $A>0$ independent of $\e$ such that for all $x,y\in \R^3$
\begin{equation}\label{ball-box}
A^{-1} d_{G,\e}(x,y) \le d_\e (x,y) \le A d_{G,\e}(x,y) .
\end{equation}
\end{lemma}
%
%
The estimate \eqref{ball-box} yields immediately
\begin{cor} The doubling property holds uniformly in $\e>0$.
\end{cor}
\begin{rmrk} Before proving \eqref{ball-box} it is useful to examine a specific example:  compare two trajectories from the origin $0=(0,0,0)$ to the point $x=(0,0,x_3)$. The first is the segment $\gamma_1$ defined by
$s\to (0,0,x_3 s)$, for $s\in [0,1]$. The length of this segment in the Riemannian metric $g^\e$ given by the orthonormal frame $X_1,X_2,\e X_3$ is $$\ell_\e(\gamma_1)=\e^{-1}|x_3|.$$ We also consider a  second trajectory $\gamma_2$ given by the subRiemannian geodesic between the two points. In view of \eqref{stesso} the length of this curve in the subRiemannian metric $g^0$ defined by the orthonormal frame $X_1,X_2$ is proportional to  $\sqrt{|x_3|}$ and coincides with the length in the Riemannian metric $g^\e$, i.e.
$$\ell_\e(\gamma_2)=\ell_0(\gamma_2)\approx \sqrt{|x_3|}.$$ Since $d_\e$ is computed by selecting the shortest path between two points in the $g^\e$ metric, then if $\sqrt{|x_3|}>\e$ one will have $d_\e(x,0) \le \sqrt{|x_3|} \approx N_\e(x)$, whereas at {\it small scales } (i.e. for $d_0(x,0)<\e$)  one will have $d_\e(x,0)\le \e^{-1}|x_3|$. By left translation invariance of $d_{G,\e}$ we have  that for any two points $x=(x_1,x_2,s)$ and $x'=(x_1,x_2,t)$, \begin{equation}\label{above}
d_\e(x,x')\le C \min ( \e^{-1}|t-s|, \sqrt{|t-s|}).
\end{equation}

From this simple example one can expect  that at large scale (i.e. for points $d^0(x,0)>\e$) the Riemannian and the subRiemannian distances are approximately the same $d_\e(x,0)\approx d_0(x,0)$.
%
\end{rmrk}
\begin{proof} From the invariance by left translations of both $d_{G,\e}$ and $d_\e$ it is sufficient to prove that $d_\e(x,0)$ and $N_\e(x)$ are equivalent. We begin  by establishing the first inequality in \eqref{ball-box}, i.e. we want to show that there exists a positive constant $A$ such that $$A^{-1}N_\e(x)\le d_\e(0,x).$$ Consider a point $x=(x_1,x_2,x_3)\in \R^3$ and three curves
\begin{itemize} 
\item A length minimizing curve $\gamma:[0,1]\to \R^3$ for the metric $g_\e$, such that $$d_\e(0,x)=\ell_\e(\gamma):=\int_0^1 \sqrt{ a_1^2 (t) + a_2^2 (t)+\e^{-2}a_3^2(t)} dt,$$ where $\gamma'(t)=\sum_{i=1,3} a_i(t)X_i|_{\gamma(t)}.$
\item An horizontal curve $\gamma_1:[0,1]\to \R^3$ with one end-point at the origin ($t=0$) and such that $\gamma_1'(t)=a_1(t) X_1|_{\gamma(t)}+a_2(t) X_2|_{\gamma(t)}$. Denote by $P=\gamma_1(1)$ and observe that $P=(x_1,x_2,p_3)$ for some value of $p_3$ such that $\int_0^1 a_3(t) dt=x_3-p_3$.
\item A vertical segment $\gamma_2:[0,1]\to \R^3$ with endpoints $P$ and $x$, such that $\gamma_2'(t)=a_3(t)X_3|_{\gamma_2(t)}$. Note that $$\e^{-1}|x_3-p_3|\le \bigg| \e^{-1} \int_0^1 a_3(t) dt\bigg| \le\int_0^1 |a_3(t)|\e^{-1}dt= \ell_\e(\gamma_2)\le \ell_\e(\gamma)\le d_\e(x,p).$$
\end{itemize}
Observe that in view of the equivalence  \eqref{stesso}, $$C^{-1}\sqrt{x_1^2+x_2^2} \le d_0(P,0)\le \ell_0(\gamma_1)=\ell_\e(\gamma_1)\le \ell_\e (\gamma),$$ for some constant $C>0$. On the other hand one also has $$\e^{-1} |x_3-p_3| \le d_\e(x,p)\le d_\e(x,0)+d_\e(0,p)\le  d_\e(x,0)+ \ell_\e(\gamma_1)\le 2d_\e(0,x).$$ Hence if $|p_3|\le \frac{1}{2} |x_3|$ then $|x_3-p_3|\ge \frac{1}{2} |x_3|$ and consequently
$$d_\e(x,0) = \ell_\e (\gamma) \ge \e^{-1} |x_3-p_3| \ge \min (\e^{-1}|x_3|, \sqrt{|x_3|}).$$
The latter yields immediately that $d_\e(x,0)\ge C^{-1} N_\e(x),$ for some value of $C>0$ independent of $\e>0$.  Next we consider the case $|p_3|>\frac{1}{2}|x_3|$. This yields
$$\min (\e^{-1}|x_3|, \sqrt{|x_3|}) \le \frac{1}{2} \min (\e^{-1}|p_3|, \sqrt{|p_3|})\le \sqrt{|p_3|}\le |P|\le Cd_0(P,0)$$
$$ \le C \ell_0(\gamma_1)=C \ell_\e (\gamma_1) \le C \ell_\e (\gamma)=C d_\e (x,0),$$ where $|P|$ is defined as in \eqref{stesso}. In summary, so far we have proved the first half of \eqref{ball-box}.

To prove the second half of the inequality we consider an horizontal segment $\Gamma_1$ joining the origin to $Q=(x_1,x_2, 0)$. Note that $d_0(0,Q)=d_\e(0,Q)=\ell_0(\Gamma_1)=\ell_\e(\Gamma_1)$. In view of \eqref{above} one has
$$d_\e(0,x)\le d_\e(0,Q)+d_\e(Q,x) = d_0(0,Q)+C\min (\e^{1} |x_3| , \sqrt{|x_3|}) \le CN_\e(x).$$ The latter completes  the proof of \eqref{ball-box}.
\end{proof}
\begin{rmrk} Similar arguments continue to hold more in general, in the setting of Carnot groups. \end{rmrk}

As a consequence of Lemma \ref{N=d}, one has that  for $\e>0$ the metric space $(\R^3, d_\e)$ is locally bi-Lipschitz to the Euclidean space, and hence its Hausdorff dimension will be $3$. As $\e\to 0$ the non-horizontal directions are {\it penalized } causing a sharp {\it phase transition } between the regime at $\e>0$ and $\e=0$.

The intuition developed through this example hints at the multiple scale aspect of the $d_\e$ metrics: At scales smaller than $\e>0$ the local geometry of the metric space
$(\R^3,d_\e)$ is roughly Euclidean; For scales larger than $\e>0$ it is subRiemannian. This intuition will inform the proofs of the stability for the doubling property in the next section.

%% file: stability.tex
\section{Stability of the homogenous structure}

The volume of Carnot-Caratheodory balls, and its doubling property, has been studied in Nagel, Stein and Wainger's seminal work \cite{NSW}. In this section we recall the main results in this paper and show how to modify their proof so that the stability of the doubling constant as $\e\to 0$ becomes evident.

\subsection{The Nagel-Stein-Wainger estimates}
Consider the Carnot-Caratheodory metric $d_\e(\cdot, \cdot)$  associated to the family of vector fields $(X_1^\e, ...,X_{p}^\e),$ defined in \eqref{definefields}. Denote by  $B_e(x,r)=\{y|d_\e(x,y)<r\}$   the corresponding metric balls.

For every $n-$tuple $I=(i_1,...,i_n)\in \{1,...,2p-m\}^n$, and for  $\bar \e \geq \e\geq 0$  define the coefficient $$\lambda^\e_I(x)=\det (X^\e_{i_1}(x),...,X^\e_{i_n}(x)).$$
For a fixed $0\le \e \le \bar \e$ and for a fixed constant $0<C_{2,\e}<1$, choose $I_\e=(i_{\e 1},...,i_{\e n})$ such that
\begin{equation}\label{bestI}|\lambda^\e_{I_\e}(x)|r^{d_\e(I_\e)} \ge C_{2,\e} max_J |\lambda^\e_J(x)|r^{d_\e(J)},\end{equation}
where the maximum ranges over all $n-$tuples.
Denote $J_\e$ the family of remaining indices, so that
$\{X^\e_{i_{\e, j}}: i_{\e, j} \in I_\e\} \cup \{X^\e_{i_{\e, k}}: i_{\e ,k} \in J_\e\}$ is the complete list $X_1^\e,...,X_{2p-m}^\e$.
When $\e=0$ we will refer to $I_0$ { as a choice corresponding to the $n-$tuple $X^0_{i_{01}},...,X^0_{i_{0n}}$ realizing
\eqref{bestI}.}
One of the main contributions in  Nagel, Stein and Wainger's seminal work \cite{NSW}, consists in the proof that
for a  $v$ and  a $x$  fixed, and letting
$$Q_\e( r)=\{u\in \R^n: |u_j| \leq  r^{d_\e( i_{\e j})}\}$$denote  a {\it weighted} cube in $\R^n$, then the quantity 
 $|\lambda^\e_{I_\e}(x)|$ provides an
estimates of the Jacobian of the exponential mapping
$u\to \Phi_{\e, v, x}(u)$ defined for   $u\in Q(r)$ as \begin{equation}\label{phi}
\Phi_{\e, v, x}(u) = exp\Big(\sum_{i_{\e, j}\in I_\e} u_j X^\e_{i_{\e, j}} + \sum_{i_{\e, k}\in J_\e} v_k X^\e_{i_{\e, k}} \Big)(x).\end{equation}

\bigskip

More precisely, for $\e\geq 0$ and fixed one has
\begin{thrm}\label{MAINNSWeps}\cite[Theorem 7]{NSW} 

 For every $\e\geq 0$,  and  $K\subset \subset \R^n$ there exist $R_\e>0$ and constants $0<C_{1, \e}, C_{2, \e} <1$  such that for every $x\in K$ and  $0<r<R_\e$, if $I_\e$ is such that \eqref{bestI}
holds, then
\begin{itemize}
\item[i)] if $|v_k| \leq C_{2 \e}r^{d(i_{\e k})}$,
$\Phi_{\e, v, x } $ is one to one on the box $Q_\e(C_{1, \e} r)$
\item[ii)] if $|v_k| \leq C_{2 \e}r^{d(i_{\e k})}$
the Jacobian matrix of $\Phi_{\e, v, x}$ satisfies on the cube $Q_\e(C_{1, \e} r)$
$$\frac{1}{4} |\lambda^\e_{I_\e} (x)| \leq |J\Phi_{\e, v, x}| \leq 4 |\lambda^\e_{I_\e} (x)| $$
\item[iii)]
$$\Phi_{\e, v, x}(Q_\e (C_{1, \e} r))  \subset B_{\e}(x,   r) \subset \Phi_{\e, v, x}(Q_\e(C_{1, \e} r/C_{2, \e})) $$
\end{itemize}
\end{thrm}

\bigskip

As a corollary one has that   the volume of a Carnot-Caratheodory ball centered in $x$ can be estimated by the measure of the corresponding  cube and the Jacobian determinant of $\Phi_{\e,v,x}$.  

\begin{cor}\label{balmeasure}(\cite[Theorem 1]{NSW})  For every $\e\geq 0$,  and  $K\subset \subset \R^n$ and for $R_\e>0$ as in Theorem \ref{MAINNSWeps}, there exist  constants $C_{3 \e}, C_{4\e}>0$  depending on $K, R_\e, C_{1,\e}$ and $C_{2\e}$ such that for all $x\in K$ and  $0<r<R_\e$ one has
\begin{equation}\label{nsw}
C_{3 \e}\sum_I |\lambda^\e_I(x)| r^{d(I)} \le |B_\e(x,r)|\le C_{4 \e} \sum_I |\lambda^\e_I(x)| r^{d(I)},
\end{equation}
\end{cor}

Estimates \eqref{nsw} in turn implies the doubling condition \eqref{D} with constants depending eventually on $R_\e, C_{1\e}$ and $C_{2\e}$.

\subsection{Uniform estimates as $\e\to 0$} Having already proved the stability of the doubling property in the special case of the Heisenberg group, in this section we turn to the general case of H\"ormander's vector fields and describe in some details results from \cite{CCR} establishing that  the constants $C_{1\e}$ $C_{2\e}$  do not vanish as $\e\to 0$.  Without loss of generality one may assume that both constants are non-decreasing in $\e$. In fact, if that is not the case 
one may consider   a new pair of constants $\tilde C_{i,\e}=\inf_{s\in [\e,\bar \e]} C_{i,s}$, for $i=1,2$.

\begin{prop}\label{indipendenza}
For every $\e\in[0,\bar \e]$, the constants $R_\e, C_{1, \e}$ and $C_{2, \e}$ in Theorem \ref{MAINNSWeps} may be chosen to be  independent of $\e$, depending only on  the $C^{r+1}$ norm of the vector fields, on the number  $\bar \e$, and on the compact $K$  .
\end{prop}

\begin{proof} The proof is split in two cases: First we study the range $\e<r<R_0$ which  roughly corresponds
to the balls of radius $r$ having a sub-Riemannian shape. In this range we show that one can select
the constants $C_{i,\e}$ to be approximately $C_{i,0}$. The second case consists in the analysis of the range
$r<\e<\bar \e$. In this regime the balls are roughly of  Euclidean shape and we show that the constants
$C_{i,\e}$ can  be approximately chosen to be $C_{i,\bar \e}$.

 Let us fix  $\e\in(0,\bar \e]$, $R=R_0$ and $r<R_0$.
We can start by describing the family $I_\e$ defined in  (\ref{bestI}), which maximize $\lambda^\e_I(x)$.
We first note that for every $\e>0$ and for every $i$, $m+1 \leq i \leq p$
we have
  \begin{equation}\label{Yer}
 X^\e_{i}r^{d_\e(i)} = \e^{d(i)-1} r Y_i, \quad X^\e_{i+p-m}r^{d_\e(i+p-m)} =  r^{d(i)} Y_i.
  \end{equation}

In the range $0<r<\e<\bar \e$ one can assume without loss  of generality  that the $n-$tuple  satisfying the maximality condition \eqref{bestI} will include only vectors of the form $\{\e^{d(i_{\e 1})-1}Y_{i_{\e 1}}, ...,
\e^{d(i_{\e n})-1}Y_{i_{\e n}}\} $ for some $n-$index $I_\e=(i_{\e 1},...,i_{\e n})$, with $1\le i_{\e k} \le p$. In fact, if this were not the case and the $n-$tuple were to include  a vector of the form $X^\e_j=Y_{j-p+m}$ for some $p<j$, then we could substitute such vector with $X^\e_{j-p+m}=Y_{j-p+m}\e^{d(j-p+m)-1}$ and from \eqref{Yer} infer that  the value of the corresponding term $|\lambda^\e_{I_\e}(x)| r^{d_\e(I_\e)}$ would increase.

Similarly, in the range $0<\e<r<\bar \e$ one can assume that the  $n-$tuple  satisfying the maximality condition \eqref{bestI} will include only vectors of the form $\{Y_{i_{\e 1}}, ...,
Y_{i_{\e n}} \}$ for some $n-$index $I_\e=(i_{\e 1},...,i_{\e n})$, with $1\le i_{\e k} \le p$. Note that the corresponding expression $$|\lambda^\e_{I_\e}(x)|r^{d_\e(I_\e)-1}=|\det (Y_{i_{\e 1}}, ...,
Y_{i_{\e n}})(x)| r^{\sum_{I_\e} d(i_{\e k})}$$ would then be one of the terms in the  left hand side of 
 \eqref{bestI} for $\e=0$, and thus is maximized by $C_{2,0}^{-1}|\lambda_{I_0}^0(x)|r^{d(I_0)-1}$.

Case 1: In view of the argument above, for every  $\e<r<R_0$ the indices $I_\e$ defined by the maximality condition (\ref{bestI}) can be chosen to coincide with indices of the family $I_0$ and do not depend on $\e$.
On the other hand the vector excluded from $I_\e$ will be not only those in $J_0$ but also the ones that have been added with a weight factor of a power of $\e$,
$$
\{X_k^\e: k\in J_\e\} = \{X^0_{i_{0,k}}: i_{0,k}\in J_0\} \cup\{\e^{d(i_{0,k})-1} X^0_{i_{0,k}}:  i_{0,k}\in I_0, i_{0,k} > m \}  $$ $$  \cup\{\e^{d(i_{0,k})-1} X^0_{i_{0,k}}:  i_{0,k}\in J_0 {, i_{0,k}>m} \}.$$
In correspondence with this decomposition of the set of indices we define a splitting in the
$v-$variables in \eqref{phi} as $$v=(\hat v, {\tilde v, \bar v}).$$
Consequently for every $\e<r$ the function  $\Phi_{\e, v, x}(u)$  can be written as  
\begin{equation}\label{e0}
\Phi_{\e, v, x}(u)   =
exp\Big(\sum_{i_{\e j}\in I_\e} u_j X^\e_{i_{\e j}} + \sum_{i_{\e k}\in J_\e} v_k X^\e_{i_{\e k}} \Big)(x)= exp\Big(\sum_{i_{0 j}\in I_0} u_j X^0_{i_{0 j}} + \sum_{i_{\e k}\in J_\e} v_k X^\e_{i_{\e k}} \Big)(x)=\end{equation}
$$exp\Big(\sum_{i_{0 j}\in I_0} u_j X^0_{i_{0 j}} + \sum_{i_{0 k}\in J_0} \hat v_k Y_{i_{0 k}} +
\sum_{i_{0 k}\in I_0, i >m } {\tilde v}_k \e^{d(i_{0 k})-1}  Y_{i_{0 k}}
+ \sum_{i_{0 k}\in J_0  {, i_{0,k}>m}  }{\bar v}_k \e^{d(i_{0 k})-1} Y_{i_{0 k}}
\Big)(x)=$$
$$=\Phi_{0, \hat v_k + \bar v_k \e^{d(i_{0 k})-1}, x}(u_1, \cdots u_m, u_{m+1} + {\tilde v}_{m+1} \e^{d(i_{0 m+1})-1}, \cdots, u_{n} + {\tilde v}_{n} \e^{d(i_{0 n})-1}).$$

Let us define mappings
$$F_{1,\e,v}(u)=\bigg(u_1,...,u_m, u_{m+1}+\tilde v_{m+1}\e^{d(i_{0  m+1})-1},...,u_{n} + {\tilde v}_{n} \e^{d(i_{0 n})-1}\bigg),$$
and
$$F_{2,\e}(v)=\bigg(\hat v_1 + \bar v_1 \e^{d( i_{01})-1 }, ..., \hat v_{2p-m} + \bar v_{2p-m} \e^{d( i_{0, 2p-m})-1}\bigg).$$
In view of \eqref{e0} we can write \begin{equation}\label{compos}
\Phi_{\e, v, x}(u) =\Phi_{0, F_{2,\e}(v), x}(F_{1,\e,v}(u)).
\end{equation}

Note that for any $\e\ge 0$ and for a fixed $v$, the mapping $u\to F_{1,\e,v}(u)$ is invertible and volume preserving in all $\R^n$. Moreover $J\Phi_{\e, v, x}(u) =J\Phi_{0, F_{2,\e}(v), x}(F_{1,\e,v}(u)).$
In view of \eqref{compos} and of Theorem \ref{MAINNSWeps}, as a function of $u$, the mapping $\Phi_{\e,v,x}(u)$ is defined,  invertible, and satisfies the Jacobian estimates in Theorem  \ref{MAINNSWeps} (ii)
$$\frac{1}{4} |\lambda^0_{I_0} (x)| \leq |J\Phi_{0, F_{2,\e}(v),x}(F_{1,\e,v}(u))|=|J\Phi_{\e, v,x}(u)| \leq 4 |\lambda^0_{I_0} (x)| $$ for all $u$ such that  $F_{1,\e,v}(u)\in Q_0(C_{1,0}r)$  and for $v$ such that
$$|F_{2,\e}^k(v)|=|\hat v_k + \bar v_k \e^{d(i_{0 k})-1}|\leq C_{2, 0}r^{d(i_{0 k})}, $$$$ |u_1|\leq  C_{1, 0}r^{d(i_{0 1})}   \cdots |u_m |\leq C_{1, 0} r^{d(i_{0 m})}, |u_{m+1} + {\tilde v}_{m+1} \e^{d(i_{0 m+1})-1}|\leq C_{1, 0} r^{d(i_{0 m+1})},$$
when $k=1,...,2p-m$.

The completion of the proof of Case 1 rests on the following two claims:

{\bf Claim 1} let  $\e<r<R_0$. There exists $C_6>0$, independent of $\e$, such that for all $v$ satisfying  $|v_k| \leq C_6r^{d(i_{\e k})}$ one has  $|F_{2,\e}^k(v)|=|\hat v_k + \bar v_k \e^{d(i_{0 k})-1}|\leq C_{2, 0}r^{d(i_{0 k})} .$

{\bf Proof of the claim:}
If we choose  $C_6< \min\{C_{1, 0}, C_{2, 0}\}$ and
$$|\hat  v_k|, |\tilde v_k|, |\bar v_k|\leq \min\{C_{1, 0}, C_{2, 0}\}\frac{r^ {d(i_{\e k})}}{4}, \quad |u_j|\leq C_{1, 0}\frac{r^{d(i_{\e j})}}{4},$$
it follows that
$$|\hat v_k|\leq C_{2, 0}\frac{r^ {d(i_{0 k})}}{4}, \quad   |\tilde v_k|, |\bar v_k|  \leq C_{1, 0}\frac{r}{4},\quad |u_j| \leq C_{1, 0}\frac{r^{d(i_{\e j})}}{4}.$$
So that
$$|\hat v_k|\leq C_{2, 0}\frac{r^ {d(i_{0 k})}}{4}, \quad \e^{d(i_{0 k})-1} |\tilde v_k|,  \quad \e^{d(i_{0 k})-1}|\bar v_k|  \leq C_{1, 0}\frac{r^{d(i_{0 k})} }{4},\quad |u_j| \leq C_{1, 0}\frac{r^{d(i_{0 j})}}{4},$$ completing the proof of the claim.

{\bf Claim 2} Let  $\e<r<R_0$ and $v$ fixed such that $|v_k| \leq C_6 r^{d(i_{\e k})}$ for $k=1,...,2p-m$. One has that $$ Q_\e(C_5^{-1} r)\subset F_{1,\e,v}^{-1}( Q_0(C_{1,0} r)) \subset  Q_\e(C_5 r)$$ for some constant $C_5>0$ independent of $\e\ge 0$.

{\bf Proof of the claim:} Choose $C_5$ sufficiently large so that $2\max \{ C_5^{-1}, C_6 \}\le C_{1,0}$ and observe that
if $u\in Q_\e(C_5^{-1} r)$ then for $k=1,...,m$ we have $|u_k|\le C_{1,0}r^{d(i_{\e,k})}=C_{1,0}r^{d(i_{0,k})}$
while for $k=m+1,...,n$ we have $|F^k_{1,\e,v}(u)|=|u_k+\tilde v_k \e^{d(i_{0k})-1}|\le
\max\{C_5^{-1}, C_6 \} r^{d(i_{0 k})} (1+\bar \e^{d(i_{0 k})-1})\le C_{1,0} r^{d(i_{0 k})}$. This proves the first inclusion in the claim. To establish the second inclusion we choose $C_5$ large enough so that $2(C_{1,0}+C_{2,\bar \e})\le C_5$ and observe that if $F_{1,\e,v}(u)\in  Q_0(C_{1,0} r)$ then
for $k=m+1,...,n$ one has $|u_k|\le |u_k+\tilde v_k \e^{d(i_{0k})-1}|+|\tilde v_k| \e^{d(i_{0k})-1}\le
2(C_{1,0}+C_{2,\bar \e}) r^{d(i_{0 k})}\le C_5 r^{d(i_{0 k})}$. The corresponding estimate for the range
$k=1,...,m$ is immediate.

In view of Claims 1 and 2, and of  Theorem \ref{MAINNSWeps}
It follows that for $\e<r$ and these choices of constants (independent of $\e$)\footnote{$R_0$ in place of $R_\e$,  $C_5$ in place of $C_{1,\e}$ and
$C_6$ in place of $C_{2,\e}$} the function $\Phi_{\e, v,x}(u) $ is invertible
on $Q_0(C_{1,0}r)$ and  i), ii) and iii) are satisfied.

\bigskip

Case 2: As remarked above, in the range $0<r<\e<\bar \e$ one can assume that the $n-$tuple  satisfying the maximality condition \eqref{bestI} will include only vectors of the form $\{\e^{d(i_{\e 1})-1}Y_{i_{\e 1}}, ...,
\e^{d(i_{\e n})-1}Y_{i_{\e n}} \}$ for some $n-$index $I_\e=(i_{\e 1},...,i_{\e n})$, with $1\le i_{\e k} \le p$.  Note that in view of \eqref{Yer} and  the maximality condition  (\ref{bestI}) the corresponding term
$$|\lambda^\e_{I_\e}(x)|r^{d_\e(I_\e)}$$
can be rewritten and estimated as follows
$$|\lambda^\e_{I_\e}(x)|r^{d_\e(I_\e)}= \e^{d(I_\e)-n} r^n |\det ( 
Y_{i_{\e 1}}, ..., Y_{i_{\e n}} ) (x)| .$$
It is then clear that the maximizing $n-$tuple $I_\e$ in (\ref{bestI})  will be identified by the lowest degree $d(I_\e)$ among all $n-$tuples  corresponding to non-vanishing determinants $\det( Y_{i_{\e 1}}, ..., Y_{i_{\e n}} ) $ in a neighborhood of the point $x$. Since this choice does not depend on $\e>r$, then one has that $I_\e=I_{\bar \e}$.
%
%
%
%
In other words, 
if we denote $$(X^{\bar \e})_{ {i_{\bar \e,k}} \in I_{\bar \e}} =
\{\bar \e^{d(i_{\bar\e,1})-1} Y_{i_{\bar\e,1}}, \cdots, \bar \e^{d(i_{\bar\e,n})-1} Y_{i_{\bar\e,n}} \}$$
then the maximality condition  (\ref{bestI}) in the range $0<r<\e<\bar \e$ can be satisfied independently from $\e$ by selecting the  family of vector fields:
 $$(X^\e)_{{i_{\e,k}} \in I_\e}=\{\e^{d(i_{\bar\e,1})-1} Y_{i_{\bar\e,1}}, \cdots, \e^{d(i_{\bar\e,n})-1} Y_{i_{\bar\e,n}} \}$$

The complementary family $J_\e$ becomes \begin{multline}
\{Y^\e_{i_{\e k}}: i_{\e k} \in J_\e\} \\ = \{\e^{d(i_{\bar \e,k})-1} Y_{i_{\bar \e,k}}: i_{0,k}\in J_{\bar \e}, \text{ with } i_{\bar \e,k} \le p \}  \cup \{ Y_{i_{\bar \e,k}-p+m}: i_{\bar \e,k}\in J_{\bar \e}, \text{ with }i_{\bar \e,k}>p\}
\end{multline}
If we denote $A_\e$, and $B_\e$ these three sets, and split the $v-$variable from \eqref{phi} as $v=(\hat v, \tilde v)$, then
it is clear that
$$Y\in A_\e \text{ iff }  \frac{{\bar \e}^{d(i_{\bar \e,k})-1} }{\e^{d(i_{\bar \e,k})-1} }Y\in A_{\bar \e}, $$ and in this case the values of $d_\e$ and $d_{\bar\e}$ are the same on the corresponding indices. Analogously $ Y\in B_\e \text{ iff } Y\in B_{\bar \e}$
and the degrees are the same.

For every $\e>r$ the map  $\Phi_{\e, v, x}(u)$  then can be written as 

$$\Phi_{\e, v, x}(u)   =
\exp\Big(\sum_{i_{\e j}\in I_\e} u_j X^\e_{i_{\e j}} + \sum_{i_{\e k}\in J_\e} v_k X^\e_{i_{\e k}} \Big)(x)= 
\exp\Big(\sum_{i_{\bar \e j}\in I_{\bar \e} } u_j X^\e_{i_{\bar \e j}} + \sum_{i_{\bar \e k}\in J_{\bar \e}} v_k X^\e_{i_{\bar \e k}} \Big)(x)$$
$$=
\exp\Big(\sum_{i_{\bar \e j}\in I_{\bar \e}} u_j \frac{{ \e}^{d(i_{\bar \e,k})-1} }{\bar \e^{d(i_{\bar \e,k})-1} }X^{\bar \e}_{i_{\bar \e j}} 
+ \sum_{i_{\bar \e k}\in J_{\bar \e}\text{ and }i_{\bar \e j} \le p}{\hat v}_k   \frac{{ \e}^{d(i_{0,k})-1} }{\bar \e^{d(i_{0,k})-1} }X^{\bar \e}_{i_{0 k}}
+ \sum_{i_{\bar \e k}\in J_{\bar \e}\text{ and }i_{\bar \e j} > p} \tilde v_k X^{\bar \e}_{i_{\bar \e k}} 
\Big)(x)$$

This function is defined and invertible for
$$ \ | \tilde v_k| ,\ |{\hat v}_k  | \frac{{ \e}^{d(i_{\bar \e,k})-1} }{\bar \e^{d(i_{\bar \e,k})-1} }\leq C_{2, \bar \e}r^{d_{\bar \e}(i_{\bar \e k})},  |u_j |\frac{{ \e}^{d(i_{0,j})-1} }{\bar \e^{d(i_{\bar \e,j})-1} }\leq C_{1, \bar \e} r^{d_{\bar \e}(i_{\bar \e j})}.$$
 Recall that with the present choice of $r<\e<\bar \e$, we have $C_{1, \bar \e} r^{d_{\bar \e}(i_{\bar \e j})}=C_{1, \bar \e} r^{d_{ \e}(i_{\bar \e j})}=C_{1, \bar \e} r^{d_{ \e}(i_{\e j})}$.
If we set
$$|\hat v_k|,| \tilde v_k|\leq C_{2, \bar \e} r^{d_{\bar \e(i_{\bar \e k})}},
$$$$|u_j|\leq C_{1, \bar \e} r^{d_{\bar \e}  (i_{\bar \e j})},$$

and argue similarly to Case 1, then the function $\Phi_{\e, v, x}$ will satisfy conditions i), ii), and iii) on $Q(C_{1,\bar \e}r)$  and hence on $Q(C_{1,\e}r)$, with constants independent of $\e$. \end{proof}

\subsection{Equiregular subRiemannian structures and equivalent pseudo-distances}
The intrinsic definition, based on a minimizing choice, of the Carnot-Caratheodory metric is not convenient when one needs to produce quantitative estimates, as we will do in the following sections. It is then advantageous to use equivalent pseudo-distances which are explicitly defined in terms of certain system of coordinates.
In the last section we have already encountered two special cases, i.e. the norms $|\cdot |$ defined in \eqref{gauge} and its Riemannian approximation \eqref{Nepsilon}. In this section we extend this construction to a all equi-regular subRiemannian structures. For $\Om\subset \R^n$ consider the subRiemannian manifold $(\Omega,\Delta, g)$ and iteratively set $\Delta^1:=\Delta $,
and $\Delta^{i+1}=\Delta^i+[\Delta^i, \Delta]$ for $i\in \N$. The bracket generating condition
is expressed by saying that there exists an integer  $s\in \N$ such that   $\Delta_{p}^s= \R^n$ for all $p\in M$.

\begin{dfn}\label{equiregular} A subRiemannian manifold $(\Om, \Delta, g)$ is {\it equiregular} if, for all $i\in \N$, the dimension of $\Delta^i_p$ is constant in $p\in \Om$. The {\it homogenous dimension}  
\begin{equation}\label{dimension} Q=\sum_{i=1}^{s-1} [\dim (\Delta_p^{i+1})-\dim (\Delta_p^i)],\end{equation}
coincides with the Hausdorff dimension with respect to the Carnot-Caratheodory distance.
\end{dfn}

This class is {\it generic} as any subRiemannian manifold has a dense open subset on which the restriction of the subRiemannian metric is equiregular.

\begin{ex} Systems of free vector fields, as defined in Definition \ref{D:free}, yield a distribution $\Delta$ that supports an equiregular subRiemannian structure for any choice of the horizontal metric $g$.
\end{ex}

Next we assume we have a equiregular subRiemannian manifold $(\Om, \Delta, g)$ and consider an orthonormal horizontal basis $X_1,...,X_m$ of $\Delta$. Following the process in \eqref{NSWspaces} one can construct  a frame $Y_1,...,Y_n$ for $\R^n$ where  $Y_1,...,Y_m$ is the original horizontal frame and $Y_{m+1},...,Y_{n}$ are commutators such that $(Y_1,...,Y_{m_k})|_p$ spans $\Delta^k_p$, for $k=1,...,s$. The degree $d(i)$ of $Y_i$ is the order of commutators needed to generate $Y_i$ out of the horizontal span, i.e. $d(i)=k$ if $Y_i\in \Delta^k_p$ but $Y_i \notin \Delta_p^{k-1}$. In particular one has $d(i)=1$ for $i=1,...,m$. The equiregularity hypothesis  allows one to choose  $Y_1,...,Y_n$ linearly independent.
Next we extend $g$ to a Riemannian metric $g_1$ on all of $T\Om$ by imposing that $Y_1,...,Y_n$ is an orthonormal basis. 

\begin{dfn}\label{metrica-epsilon} For any $\e\in (0,\bar \e]$ we define the Riemannian metric $g_\e$ by setting that $\{ \e^{d(i)-1} Y_i$, $i=1,...,n$ is an orthonormal frame. Denote by $d_\e(x,y)$ the corresponding Riemannian distance function.
\end{dfn}
\begin{rmrk} Repeating the proof of \cite[Theorem 4]{NSW} one immediately sees that $d_\e$ as defined here  is comparable to the distance $d_\e$  defined in  Section \ref{DEPS}, with equivalence constants independent of $\e>0$.
\end{rmrk}
We define canonical coordinates around a point $x_0\in \Om$ as follows. Since $Y_1,...,Y_n$ is a generating frame for $T\Om$ then for any point  $x$ in a neighborhood $\omega$ of $x_0$ one has that there exists a unique $n-$tuple $(x_1,...,x_n)$ such that \begin{equation}\label{x-coord}
\exp( \sum_{i=1}^n x_i Y_i)(x_0)=x.\end{equation}
We will set $x=(x_1,...,x_n)$ and use this $n-$tuple as local coordinates in $\omega$.

\begin{dfn} For every $x=(x_1,...,x_n)\in \omega$ we define a pseudo-distance
$d_{G,\e}(x,x_0):=N_\e(x_1,...,x_n)$ with 
\begin{equation}
N_\e(x_1,...,x_n):= \sqrt{\sum_{i=1}^m x_i^2 }+ \sum_{i=m+1}^n \min  \big(\e^{-(d(i)-1)} |x_i|, |x_i|^{1/d(i)}  \big).
\end{equation}
For $\e=0$ we set 
$$N_0(x_1,...,x_n):=\sqrt{\sum_{i=1}^m x_i^2 }+   \sum_{i=m+1}^n  |x_i|^{1/d(i)}  .$$
\end{dfn}
\begin{thrm}\label{epsilon-equiv} For every compact $x_0\in K\subset \omega$ there exists $C=C(K,\Delta, g, \omega)>0$, independent of $\e\in (0,\bar \e]$ ,  such that
$$C^{-1} d_{G,\e}(x,x_0) \le d_{\e}(x,x_0) \le C d_{G,\e}(x,x_0)
$$
for all $x\in K$.
\end{thrm} 

\begin{rmrk} Note that for $\e=0$ the equivalence is a direct consequence of the {\it Ball-Box theorem} proved by Nagel, Stein and Wainger \cite{NSW} or Mitchell \cite[Lemma 3.4]{mitchell}. This observation replaces the estimates \eqref{stesso} from the Heisenberg group setting.
\end{rmrk}

The proof of Theorem \ref{epsilon-equiv} follows as a corollary of the following
\begin{prop} In the hypothesis of Theorem \ref{epsilon-equiv} one has that there exists $R=R(K,\Delta, g, \omega)>0, C=C(K,\Delta, g, \omega)>0$, independent of $\e\in (0,\bar \e]$ ,  such that for all $x\in K$ and $r\in (0,R)$,
$$ B_{G,\e}(x_0, C^{-1} r) \subset B_{\e}(x_0,r) \subset  B_{G,\e}(x_0, C r),
$$
where $$B_{G,\e}(x_0,  r) :=\Bigg\{ x\in \R^n \text{ such that } \max_{i=1,...,s} \Bigg[  \min  \big(\e^{-(d(i)-1)} |x_i|, |x_i|^{1/d(i)}  \big)  \Bigg]   < r  \Bigg\}.$$

\end{prop}
\begin{proof}  The proof follows closely the arguments in the previous section and is based on the results   in \cite{NSW}.  In view of the equiregularity hypothesis note that $Y_1,...,Y_n$ are linearly independent and the construction in \eqref{definefields} yields the distribution
$X_1^\e,...,X_{2n-m}^\e$ over $\Omega$.  Recall from \eqref{phi}, Proposition \ref{indipendenza} and Theorem \ref{MAINNSWeps} that  if $I_\e, J_\e$ are chosen as in \eqref{bestI} and for any $v=(v_1,...,v_{n-m})$ such that $|v_k| \leq C_{2 \e}r^{d(i_{\e k})}$, one has
\begin{equation}\label{equiv-balls}B_\e(x_0,r) \approx \Phi_{\e,v,x_0} (Q_\e (r)),\end{equation}
with constants independent from $\e\ge 0$, 
where  $Q_\e=\{u\in \R^n: |u_j| \leq  r^{d_\e( i_{\e j})}\}$,   and
$$\Phi_{\e, v, x}(u) = exp\Big(\sum_{i_{\e, j}\in I_\e} u_j X^\e_{i_{\e, j}} + \sum_{i_{\e, k}\in J_\e} v_k X^\e_{i_{\e, k}} \Big)(x).$$
The $n-$tuple $I_\e$  contains $n$ indexes related either  to the horizontal vector fields $X_1^\e, ..., X_m^\e$ or to the  commutators $X_{m+1}^\e,...,X_n^\e$. The latter may consist of {\it weighted } versions $X_{m+1}^\e,..., X_n^\e$ or {\it unweighted versions} $X_{n+1}^\e, ..., X_{2n-m}^\e$.  In either case the same vector will appear both in the weighted and in the unweighted version (either among the $I_\e$ indexes or in the complement $J_\e$).  Comparing the representation  $\Phi_{\e,v,x_0}$ with the  $x-$coordinates representation \eqref{x-coord}  one has $$\exp( \sum_{i=1}^n x_i Y_i)(x_0)= exp\Big(\sum_{i_{\e, j}\in I_\e} u_j X^\e_{i_{\e, j}} + \sum_{i_{\e, k}\in J_\e} v_k X^\e_{i_{\e, k}} \Big)(x_0),$$ and we let for each $k=1,...,n$
$$x_k= \begin{cases}  \e^{d(k)-1} u_{i_k} + v_{j_k}  & \text{ if }   i_k\le n \\
u_{i_k}+ \e^{d(k)-1}   v_{j_k}   & \text{ if }  i_k>n \end{cases}. $$ From the latter we obtain that for all $k=1,...,n$
$$|x_k| \le C(\e^{d(k)-1} r + r^{d(k)}).$$

If $x\in B_\e(x_0,r)$ then $|u_{i_k}|, |v_{j_k}|\le C r^{d(k)}$. Consequently,  
$$\min ( \e^{-(d(k)-1)} |x_k|, |x_k|^{1/d(k)} ) \le C \min ( \e^{-(d(k)-1)} |\e^{d(k)-1} r + r^{d(k)}|, \bigg[ \e^{d(k)-1} r + r^{d(k)}\bigg]^{1/d(k)} )$$
$$\le C \min ( r \bigg[ 1+ \bigg(\frac{r}{\e}\bigg)^{d(k-1)}\bigg], \ \ r \bigg[ \bigg(\frac{\e}{r}\bigg)^{d(k)-1} + 1\bigg]^{1/d(k)} ) \le 2C r.$$
This shows that for $r>0$ sufficiently small, and for some choice of $C>0$ independent of $\e\ge 0$, we have  $B_{\e}(x_0,r) \subset  B_{G,\e}(x_0, C r).$

To prove the reverse inclusion we consider a point $x=\exp( \sum_{i=1}^n x_i Y_i)(x_0)\in   B_{G,\e}(x_0, C r)$. Select $I_\e$ as in \eqref{bestI} and set $v=0$ to  represent $x$ in the basis $X_{i_1},...,X_{i_n}$ as
$$x=exp\Big(\sum_{i_{\e, j}\in I_\e} u_j X^\e_{i_{\e, j}} \Big)(x_0).$$
In view of Theorem \ref{MAINNSWeps}, and \eqref{equiv-balls},  to  prove the proposition it suffices  to show that there exists a constant $C>0$ independent of $\e>0$ such that for each $j=1,...,n$ one has $|u_j|\le Cr^{d_\e(i_{\e j})}$.

We distinguish two cases: In the range $\e\ge 2r$ one can argue as in \eqref{Yer} to deduce that for each $j=1,...,n$  we may assume without loss of generality that the contribution due to $u_jX^\e_{i_{\e, j}}$
follows from the choice of a weighted vector, and hence is of the form
$u_j \e^{d(k)-1} Y_k$ for some $k>m$. Consequently one has $d_\e(i_{\e, j})=1$ and $x_k=u_j \e^{d(k)-1}$. 

On the other hand, since $\e\ge 2r $ then one must also have that 
$$\min ( \e^{-(d(k)-1)} |x_k|, |x_k|^{1/d(k)} ) = \e^{-(d(k)-1)} |x_k| <r.$$
Consequently one has
$$|u_j|= |x_k| \e^{1-d(k)} \le r= r^{d_\e(i_j)}.$$

In the range $\e<2r$ we observe that one must have $|x_k|\le Cr^d(k)$.
Arguing as in \eqref{Yer} we see that without loss of generality, or each $j=1,...,n$,  the contribution due to $u_jX^\e_{i_{\e, j}}$
follows from the choice of a un-weighted vector, and hence is of the form
$u_j  Y_k$ for some $k>m$. Consequently one has $d_\e(i_j)=d(k)>1$ and $x_k=u_j $, concluding the proof.

\end{proof}

%% file: HeatKernel.tex
\subsection{H\"ormander type parabolic operators in non divergence form}
The results in this section concern uniform Gaussian estimates for the heat kernel of certain degenerate parabolic differential equations, and their parabolic regularizations. We will consider a collection of smooth vector fields $X=(X_1, \cdots, X_m)$ satisfying 
H\"ormander's finite rank condition (\ref{Hor}) in an open set $\Om\subset \R^n$. We will   use throughout the section  the definition of degree $d(i)$ relative to the stratification \eqref{NSWspaces}.

A second order, non-divergence form,  ultra-parabolic operator with 
constant coefficients $a_{ij}$ can be expressed as:
\begin{equation}\label{operator}L_A = \p_t - \sum_{i,j=1}^ma_{ij} X_i X_j ,\end{equation}
where $A=( a_{ij})_{ij=1, \ldots m}$  is a symmetric,  real-valued, positively defined $m\times m$ matrix satisfying 
\begin{equation} \label{unifellip}\Lambda^{-1} \sum_{d(i)=1} \xi_i^2 
 \leq \sum_{i,j=1}^m a_{ij} \xi_i \xi_j \leq \Lambda \sum_{d(i)=1} \xi_i^2 \end{equation}
for a suitable constant $\Lambda$. We will also call 
\begin{equation} M_{m, \Lambda}\text{ the set of symmetric  } m\times m \text{ real valued matrix, satisfying \eqref{unifellip}}\end{equation}
If $A$ is the identity matrix then the existence of a heat kernel 
for the operator $L_A$ is a by now classical result due to Folland \cite{fol:1975} and Rothschild and Stein \cite{Roth:Stein}. Gaussian 
estimates have been provided by Jerison and Sanchez-Calle \cite{MR865430}, and by Kusuoka and 
Strook \cite{Kusuoka}.
There is a broad, more recent literature dealing with Gaussian estimates for  non divergence form operators  with H\"older continuous coefficients 
$a_{ij}$. Such estimates have been systematically studied in \cite{BONFI}, \cite{BLU}, \cite{BLU-2004}, \cite{BBLU} where a self-contained proof is provided. 

A natural technique for studying the properties of the operator $L_A$ 
is to consider a parabolic regularization
induced by  the vector fields $X_i^\e $ defined 
in (\ref{definefields}). More precisely, we will define the operator
  \begin{equation}\label{operatore}L_{\epsilon, A} = \p_t - \sum_{i,j=1}^p a^\e_{ij} X_i^\e X_j^\e  \end{equation}
where $a^\e_{i,j} $ is any $p\times p$ positive defined matrix belonging to $ M_{p, 2 \Lambda}$ and such that 
$$a^\e_{i,j} = a_{i,j} \quad \text{ for }i,j=1, \dots m.$$
We will denote 
\begin{equation}\label{matrixset}
M^\e_{p, 2\Lambda}
\end{equation}
the set of such matrices. Formally, the operator $L_{A}$ can be recovered as a
 limit as $\e\to 0$  of operator $L_{\e, A}$. 
Here we are interested in understanding which are the properties of solutions of $L_{\e, A}$ which are preserved in the limit.  

For $\e>0$ consider 
a Riemannian metric $g_\e$ defined as in Remark \ref{subR}, such that  the vector fields $X_i^\e$ are orthonormal. The induced distance function $d_\e$
is biLipschitz equivalent to  the Euclidean norm $||_E$. 
Consequently, the operator $L_{\epsilon, A}$ has a fundamental solution $\Gamma_{\e, A}$, which can be estimated as 
  \begin{equation}\label{gammae}\Gamma_{\e, A}(x)\leq C_\e\frac{e^{-\frac{|x|_{E}^2}{C_\e t}}} {t^{n/2}}\end{equation}
  for some positive constant $C_\e$ depending on $A,\e$ and $X_1,...,X_m$.

Unfortunately the constant $C_\e$ blows up as $\epsilon$ approaches $0$, so the Riemannian estimate \eqref{gamma} alone
does not provide Gaussian bounds of the fundamental solution $\Gamma_A$ of the limit 
operator (\ref{operator}) as $\epsilon$ goes to $0$.  In \cite{Krylov} the elliptic regularization 
technique has been used to obtain 
$L^p$ and $C^\alpha$ regularity of the solutions, which however are far from being 
optimal. 
In \cite{CiMa-F}, new estimates uniform in $\e$ have been provided, in the time independent setting 
which are optimal with respect to 
the decay of the limit operator. In 
\cite{MR3108875} the result has been extended to the parabolic operators, in the special case of Carnot groups.

\bigskip
In order to further extend these estimates, we need to formulate the following definition: 

\begin{dfn}\label{defkernelestimates}

We say that a family of kernels $(P_{\e, A})_{\e>0, A\in M^\e_{p, 2\Lambda}}$, defined on 
$\R^{2n}\times ]0, \infty[ $ has, on the compact sets of an open set $\Om$, an    exponential decay of order $2 + h $,  
uniform with respect to a family of distances $(d_\e)_\e$ and of matrices
$A\in M^\e_{p, 2\Lambda}$ (see definition \ref{matrixset}) 
and we will denote $P_{\e, A}\in \mathcal{E}(2+h, d_\e, M^\e_{p, 2\Lambda}) $ if the following three condition hold: 
\begin{itemize}
\item{For every $K\subset\subset\Omega$ 
there exists a constant $C_\Lambda>0$  depending on $\Lambda$ but independent of $\e>0$, and of the matrix 
$A\in  M^\e_{p, 2\Lambda}$ such that for each $\e>0$, $x,y\in K$ and $t>0$ one has
\begin{equation}\label{gamma}
C_\Lambda^{-1} \frac{t^\frac{h}{2}e^{-C_\Lambda\frac{d_\e(x,y)^2}{t}}} {|B_\e (x, \sqrt{t})|}\le P_{\e, A}(x,y,t)\le C_\Lambda\frac{t^\frac{h}{2} e^{-\frac{d_\e(x,y)^2}{C_\Lambda t}}} {|B_\e (x, \sqrt{t})|}.
\end{equation}}
\item{For  $s\in \N$ and $k-$tuple $(i_1,\ldots,i_k)\in \{1,\ldots,m\}^k$ there exists a constant $C_{s,k}>0$ depending only on $k,s,X_1,...,X_m,\Lambda$ such that
\begin{equation}\label{Xgamma}
|(\p_t^s X_{i_1}\cdots X_{i_k} P_{\e, A})(x,y,t)| \le C_{s,k}    \frac{t^\frac{h-2s-k}{2} e^{-\frac{d_\e(x,y)^2}{C_\Lambda t}}} {|B_\e(x, \sqrt{t})|}
\end{equation}
for all $x,y\in K$ and $t>0$.}
\item{For any $A_1,A_2\in M_\Lambda$, $s\in \N$ and $k-$tuple $(i_1,\ldots,i_k)\in \{1,\ldots,m\}^k$ there exists $C_{s,k}>0$ depending only on $k,s,X_1,...,X_m, \Lambda$ such that
\begin{equation}\label{AmenA}
|(\p_t^s X_{i_1}\cdots X_{i_k} P_{\e, A_1})(x,y,t) - \p_t^s X_{i_1}\cdots X_{i_k} P_{\e, A_2})(x,y, t) |\leq
\end{equation}
$$ \le ||A_1 - A_2||C_{s,k}  \frac{t^\frac{h-2s-k}{2} e^{-\frac{d_\e(x,y)^2}{C_\Lambda t}}} {|B_\e(x, \sqrt{t})|},$$ where $||A||^2:=\sum_{i,j=1}^n a_{ij}^2$.}
\end{itemize}
\end{dfn}

With these notations we will now extend all  these previous results to vector fields which only satisfy the H\"ormander condition, establishing 
estimates which are uniform in the variable $\e$ as $\e\to 0$, and in the choice of the 
matrix $A\in M^\e_{2\Lambda}$ for the fundamental solutions 
$\Gamma_{\e, A}$ of the operators 
 $L_{\e, A}$. To be more specific, we will prove: 

\begin{prop}\label{uniform heat kernel estimates}
The fundamental solution $\Gamma_{\e, A}$ of the operator 
 $L_{\e, A}$, is a kernel with  exponential 
decay of order $2$, uniform with respect to $\epsilon>0$ 
and to $A\in M^\e_{m,\Lambda}$, according to definition (\ref{defkernelestimates}). 
Hence it belongs to the set  $\mathcal{E}(2, d_\e, M^\e_{2\Lambda}) $. 
Moreover, if $\Gamma_A$ is the fundamental solution of the operator 
$L_\A$ defined in (\ref{operator}) one has
\begin{equation}\label{GetobarG}{ X}^\e_{i_1}\cdots { X}^\e_{i_k} \p_t^s  \Gamma_{\e, A}\to {X}_{i_1}\cdots {X}_{i_k}\p_t^s \Gamma_{A}\end{equation}
as $\e\to 0$  uniformly on compact sets and in a dominated way  on subcompacts of $\Om$.

\end{prop}

Our main contribution is  that all the 
constants are independent of $\e$. The proof of this assertion is based on a lifting procedure, which 
allows to express the fundamental solution of the operator $L_{A, \e}$ in terms 
of the fundamental solution of a new operator ${\bar L}_{A}$ independent of $\e$. 
The lifting procedure is composed by a first step in which we apply 
the delicate Rothschild and Stein lifting technique \cite{Roth:Stein}. 
After that, when the vector fields are free up to a specific step, 
we apply a second lifting which has been introduced in \cite{CiMa-F}, 
where the time independent case was studied, and  
from \cite{MR3108875} where the Carnot group setting is considered.

The simplest example of such an equation is the Heat equation associated to the Kohn Laplacian in the Heisenberg group, $\p_t - X_1^2 -X_2^2$, where the  vector fields $X_1$ and $X_2$ have been expressed on coordinates in Example \ref{heisenberg-ex}. In order to present  our approach we will give an outline of the proof in this special setting. 

\begin{ex} Denote by $(x_1, x_2, x_3)$  points of $\R^3$, let  $X_1, X_2, X_3$ be the vector fields defined in Example \ref{heisenberg-ex}, and let $I$ denote the identity matrix:
Consider the parabolic  operator $$L_{\e, I} = - \partial_t + X_1^2 + X_2 ^2 +  \e^2 X_3^2,$$
and note that it becomes degenerate parabolic as $\e\to 0$. Let $d_\e$ denote the Carnot-Caratheodory distance associated to the distribution $X_1,X_2, \e X_3$.

In order to handle such degeneracy we introduce new variables $(z_1, z_2, z_3)$ and a new set of vector fields replicating the same structure of the initial ones, i.e., 
$$\hat Z_1= \partial_{z_1} + z_2 \partial_{z_3}, \hat Z_2= \partial_{z_2} - z_1 \partial_{z_3}, \hat Z_3 = \partial_{z_3}$$
with $(x_1,x_2,x_3, z_1,z_2,z_3)\in \Hone \times \Hone$.
The next step consists in lifting $L_{\e,I}$ to an operator 
$$ \bar L_\e = \partial_t + X_1^2 + X_2 ^2 +  Z_1^2 + Z_2^2 + (Z_3 + \e X_3)^2 ,$$  defined on  $\Hone\times \Hone$, and denote by
$\bar \Gamma_\epsilon$ its  fundamental solution. Let $\bar d_\e$ denote the Carnot-Caratheodory distance generated by $X_1,X_2,Z_1,Z_2, (Z_3+\e X_3)$ and arguing as in \eqref{more-equivalence} note  that $\bar d_\e ((x,z), (y,z))\ge d_\e (x,y)-C_0$, for some constant $C_0$ independent of $\e$.
Consider the change of variables on the Lie algebra  of   $\Hone \times \Hone$,
$$X_i\rightarrow X_i, Z_i\rightarrow Z_i, \text{ for }i=1,2, Z_3 + \e X_3\rightarrow Z_3. $$ Note that the Jacobian of such change of variables does not depend on $\e$ and that it reduces the  operator $\bar L_\e$ to 
$$ \bar L = \partial_t + X_1^2 + X_2 ^2 +  Z_1^2 + Z_2^2 + Z_3^2 $$ 
whose  fundamental solution we denote by $\bar \Gamma$. 
Note that this operator is parabolic with respect to the vector fields $Z_i$ and degenerate parabolic with respect to the vector fields $X_i$. 
Is is clear that the operator $ \bar L $ is independent of $\e$,  and consequently its fundamental solution $ \bar \Gamma$ satisfies standard Gaussian estimates with constants independent of $\e$
$$\bar \Gamma(x,t)\le C_\Lambda\frac{e^{-\frac{\bar d (x,0)^2}{C_\Lambda t}}} {|\bar B (0, \sqrt{t})|},
$$
where $\bar d$ denotes the Carnot-Caratheodory distance in $\Hone\times \Hone$ generated by the distribution of vector fields $X_1,X_2,Z_1,Z_2,Z_3$.
Changing back to the original variable we see that also $\bar\Gamma_\e$ satisfies 
analogous estimates with the same constants, with the distance $\bar d$ replaced by the   distance $\bar d_\e$ naturally associated to the operator $\bar L_\e$. 
Finally, integrating with respect to the added variable $(z_1, z_2, z_3)$, we obtain an uniform bound for the fundamental solution of the operator $L_{\e,I}$ in terms of the distance $d_\e$.
\end{ex}

\subsection{The Rothschild-Stein freezing and lifting theorems}
Let us first recall a local lifting procedure introduced by Rothschild and Stein in  \cite{Roth:Stein} 
which, starting from a family $(X_i)_{i=1, \cdots m}$ 
of H\"ormander type vector fields of step $s$ in a neighborhood of $\R^n$, leads  to the construction of 
a new family of vector fields which are free, and of H\"ormander type with the same step $s$, in a neighborhood of a larger space. The projection of the new free vector fields on $\R^n$ yields the original vector fields, and that is why they are called {\it liftings}.
 
Let us start with some definitions:
\begin{dfn}\label{D:free}
Denote by $n_{m,s}$ the dimension (as a vector space) of the free
nilpotent Lie algebra with $m$ generators and step $s$. Let  $X_1,\ldots,X_m$ be  a set
of smooth vector fields defined  in an open neighborhood of a
point $x_0 \in \R^n$, and let 
$$V^{(s)} = span \{X^{(1)}, \cdots, X^{(r)}\},$$
where the sets $X^{j}$ are as defined in (\ref{NSWspaces}). 
We shall
say\index{free vector fields} that $X_1,\ldots,X_m$ are {\it free up
  to step $s$} if for any $1\le r\le s$ we have $n_{m,s}= dim(V^{(s)})$.
\end{dfn}

If a point $x_0 \in R^n$ is fixed, the lifting procedure of Rothschild-Stein locally introduces new variables $\tilde z $ and new vector fields $(\tilde Z_i)$ expressed in terms of the new variables such that in a neighborhood $U$ of  $x_0$ the vector fields 
$\tilde X_i = (X_i +  \tilde Z_i)_{i=1,\ldots,
m}$ are free at step $s$. More precisely, one has  \cite[Theorem 4]{Roth:Stein} 

\begin{thrm}\label{RS-lift}
Let $X_1,\ldots,X_m$ be a system of smooth vector fields,
satisfying (\ref{Hor}) in an open set $U\subset \R^n$. For any $x \in
U$ there exists a connected open neighborhood of the origin  $V\subset
\R^{\nu- n}$, and  smooth functions $\lambda_{ij}(x,\tilde z)$, with
$x\in R^n$ and $\tilde z=(z_{n+1},\ldots,z_{\nu})\in V$, defined in a
neighborhood $\tilde{U}$ of $\tilde{x}=(x,0) \in U\times V \subset \R^\nu$,
such that the vector fields $\tilde{X}_1,\ldots,\tilde{X}_m$ given
by
$$\tilde{X}_i=X_i +  \tilde Z_i, \quad \tilde Z_i = \sum_{j=n+1}^{\nu} \lambda_{ij}(x,\tilde z)\p_{z_j}$$
are free up to step $r$ at every point in $\tilde{U}$.
\end{thrm} \index{Rothschild-Stein lifting theorem}

\begin{rmrk} 
In the literature the lifting procedure described above is often coupled with another key result introduced in \cite{Roth:Stein}, a nilpotent approximation which is akin to the classical {\it freezing} technique for elliptic operators. Let us explicitly note that in section 5.3  we need only to apply the lifting theorem mentioned above, and not the freezing procedure.  In particular, in the example of the so called Grushin vector fields
$$X_3=\p_{x_1} \text{ and }X_4=x_1 \p_{x_2}$$
they would need to be lifted through this procedure to the Heisenberg group structure 
$$X_3=\p_{x_1} \text{ and }X_4=\p_{x_3}+x_1 \p_{x_2}.$$
On the other hand the vector fields 
$$X_1=\cos\theta \p_{x_1}+\sin \theta \p_{x_2}  \quad \text{and} \quad  X_2=\p_{\theta}$$
will be unchanged by the lifting process, since they are already  free up to step 2. 
\end{rmrk} 

Later on, In section 5.4 we will apply  Rothschild and Stein's freezing theorem to a family of vector fields 
$X_1,\ldots,X_m$ free up to step $r$. This will allow 
 to approximate a 
given family of vector fields with homogeneous ones. 
Note that in this case the function $\Phi$ in (\ref{phi}) is independent of $v$ and 
its expression reduces to: 
\begin{equation}\label{phi0}
\Phi_{x}(u) = exp\Big(\sum_{i}   u_i  X_{i} \Big)(x).\end{equation}
The pertinent theorem from \cite{Roth:Stein} is the following,

\begin{thrm}\label{rs-5}
Let  $X_1,\ldots,X_m$ be a family of vector fields 
are free up to rank $r$ at every point. 
Then for every $x$ there exists a neighborhood $V$ of $x$ and a neighborhood $U$
of the identity in $\bG_{m,r}$, such that:
\begin{enumerate}
\item[(a)] the map $\Phi_{x}:
U\to V $ 
is a diffeomorphism onto its image. We will call $\Theta_x$ its inverse map
\item[(b)]we have
\begin{equation}\label{split}
d \Theta_x (X_i)= Y_i+R_i, \ \ i=1,\ldots,m
\end{equation}
where $R_i$ is a vector field  of local degree less or equal than
zero, depending smoothly on $x.$
\end{enumerate}
Hence the operator $R_i$ will represented in the form:
$$R_i= \sum_{jh} \sigma_{i}(u) X_i,$$
where $\sigma$ is an homogeneous polynomial of degree  $d(X_i)-1.$
\end{thrm} \index{Rothschild-Stein osculating theorem}

\subsection{A lifting procedure uniform in $\epsilon$}

So far we have started with a set of H\"ormander vector fields $X_1,...,X_m$ in $\Om\subset \R^n$ and we have lifted them through Theorem \ref{RS-lift} to a set $\tilde X_1,...,\tilde X_m$ of H\"ormander vector fields that are free up to a step $s$ in a neighborhood $\tilde \Om\subset \R^\nu$. Next, we perform  a second lifting inspired by the work in \cite{MR3108875}. 
 We will consider the augmented space   $ \R^{\nu}\times \R^\nu$  defined in terms of 
 $\nu$ new coordinates $\hat z=(\hat z_1,...,\hat z_\nu)$. Set $z=(\tilde z, \hat z)$ and   denote points of $\R^{\nu}\times \R^{\nu}$ by $\bar x =(x, \tilde z, \hat z) =(x,z)$.  Denote
by $\hat Z_1,\ldots, \hat Z_m$ a family of vector fields free up to step s.  $\tilde X_1,...,\tilde X_m$, i.e. a family of vector fields free of step $s$ in the variables $\hat z$, and let 
$$\hat Z_{m+1}, \cdots \hat Z_\nu$$
denote the complete set of their commutators, as we did in \eqref{NSWspaces}.
Note that the subRiemannian structure generated by $\hat Z_1,...,\hat Z_m$ coincides with the structure generated by the family $\tilde  X_i$, but are defined in terms of new variables $\hat z$.

For every $\e\in [0,1)$ consider a 
sub-Riemannian structure determined by the choice of horizontal vector fields given by
\begin{equation}\label{frame2}
(\bar X_1^\e, \cdots \bar X^\e_{m + \nu}) = (\tilde X_1,\ldots, \tilde X_m, \hat Z_1,\ldots, \hat Z_m, \tilde X_{m+1}^\e+\hat Z_{m+1},\ldots, \tilde X_{\nu}^\e+\hat Z_{\nu}).\end{equation}

Since the space is free up to step $r$ the function $\Phi$ in (\ref{phi}) is independent of $v$ and 
its expression reduces to: 
\begin{equation}\label{phi}
\Phi_{\e, \bar x}(u) = exp\Big(\sum_{i}   u^\e_i \bar X^\e_{i} \Big)(\bar x).\end{equation}
In the sequel, when we will need to explicitly indicate the vector fields defining $\Phi$ we 
will also use the notation: 
\begin{equation}\label{phivec}
\Phi_{\e, \bar x, \bar X^\e}(u) = \Phi_{\e, \bar x}(u), \text{ and } \Phi_{\e, \bar x, \bar X^\e i}(u) 
\text{ its components,}
\end{equation}
and analogous notations will be used for the inverse map $\Theta_{\e, \bar x, \bar X^\e}$

For every $\e> 0$ and  $\bar x, \bar x_0$, in view of Theorem \ref{epsilon-equiv}  the associated ball box distances reduce to: 

$$\bar d_\e(\bar x, \bar x_0)= \sum_{i=1}^{2m}|u^\e_i|  + \sum_{i=2m+1}^{\nu+m}  \min(|u^\e_i|, |u^\e_i|^{1/d(i)} )+ \sum_{i=\nu+m+1}^{2\nu}| u^\e_i|^{1/d(i )}$$
For $\e=0$ and  $\bar x, \bar x_0 $  we have
$$\bar d_0(\bar x, \bar x_0)= \sum_{i=1}^{n}  | u^0_i|^{1/d(i )}$$

\subsection{Proof of the stability result}

The sub-Laplacian/heat operator associated to this structure is 
$$
\bar L_{\e,A}=\p_t -\sum_{i=1}^{m+ \nu} \bar a_{ij} \bar X^\e_{i}\bar X^\e_{j},$$
where 
$$\bar A= A \oplus \lambda I$$
and $I$ is the identity matrix of dimension $\nu \times \nu$.  
We denote by    $\bar \Gamma_{\e, A}$ the heat kernels of the corresponding heat operators, and prove a 
lemma analogous to lemma \ref{uniform heat kernel estimates} for the lifted operator:

\begin{lemma}\label{convergencelem1}
The fundamental solution $\bar\Gamma_{\e, A}$ of the operator 
 $\bar L_{\e, A}$, is a kernel with local uniform exponential 
decay of order $2$  with respect to $\epsilon>0$ 
and $A\in M^\e_{m+\nu,\Lambda}$, according to definition (\ref{defkernelestimates}). 
Hence it belongs to the set  $\mathcal{E}(2, \bar d_\e, M^\e_{m+\nu,\Lambda}) $. 
Moreover, as $\e\to 0$ one has
\begin{equation}\label{barGetobarG}{ X}^\e_{i_1}\cdots { X}^\e_{i_k} \p_t^s \bar \Gamma_{\e, A}\to {X}_{i_1}\cdots {X}_{i_k}\p_t^s \bar\Gamma_{A}\end{equation}
uniformly on compact sets, in a dominated way on all $\bar G$.
\end{lemma}

\begin{proof}
The result for the limit operator $\bar L_{0,A} $ is well known and contained for example in \cite{BBLU}. 
Hence we only have to estimate the fundamental solution of the operators $\bar L_{\e,A}$ in terms of the one of 
$\bar L_{0,A} $. In order to do so, we first define a change of variable on the Lie algebra: 
\begin{equation}\label{te}T_\e (\bar X_i^\e)=\bar X^0_i\text{ for } i=1,\ldots,\nu +m\end{equation}
Then from a fixed point $\bar z$ we apply the exponential map to induce on the Lie group 
a volume preserving change of variables. 
Using the notation introduced in (\ref{phivec}),  we will denote
$$\bar F_{\e, \bar z} :\bar G\to \bar  G, \quad \bar F_\e (\bar x)= \exp( \Phi^{-1}_{\e, \bar z, T_\e(\bar X^0) i} (\bar x) \bar X^0_i)(\bar z) $$
Since the distances are defined in terms of the exponential maps, this change of variables 
induces a relation between the distances $\bar d_0$ and $\bar d_\e$: 
\begin{equation}\label{distances}\bar d_\e (\bar x, \bar x_0)=\bar d_0( \bar F_\e (\bar x), \bar  F_\e(\bar x_0)).\end{equation}
Analogously we also have
\begin{equation}\label{changegamma}\bar \Gamma_{\e, A} (\bar x, \bar y, t)=\bar \Gamma_{0,A}(\bar F_\e(\bar x), \bar F_\e(\bar y), t),\end{equation}

Hence assertions (\ref{gamma}) follow from the estimates of $\bar\Gamma_{0,A}$ contained for  instance in \cite{MR865430}. Indeed the second inequality can be established as follows: 
$$\bar \Gamma_{\e, A} (\bar x, \bar y, t)=\bar \Gamma_{0,A}(\bar F_\e(\bar x), \bar F_\e(\bar y), t) \leq
 C_\Lambda\frac{e^{-\frac{\bar  d_0(\bar F_\e(\bar x), \bar F_\e(\bar y))^2}{C_\Lambda t}}} {|\bar  B_0 (\bar F_\e(\bar x), \sqrt{t})|}= 
 C_\Lambda\frac{e^{-\frac{\bar  d_\e(\bar x, \bar y)^2}{C_\Lambda t}}} {|\bar  B_\e (\bar x, \sqrt{t})|}.$$

The proof of the first inequality in (\ref{gamma}) and (\ref{Xgamma}) is analogous, 
while (\ref{AmenA}) follows from the estimates of the fundamental solution contained in (\cite{BBLU}).

The pointwise convergence (\ref{barGetobarG}) is also an immediate consequence of  \eqref{distances} and \eqref{changegamma}.
In order to prove the dominated convergence result we need to relate the distances $\bar d_0$ and $\bar d_\e$. 
On the other side, the change of variable (\ref{te}) allows to express exponential coordinates 
$u_i^\e,$ in terms of $u_i^0$ as follows:
 $$\bar d_\e( \bar x,  \bar x_0)= \sum_{i=1}^{2m}|u_i^0|  + \sum_{i=2m+1}^\nu \Big( |u_i^0- \e w_{i+\nu}^0|^{1/d(i)} + \min(|u_i^0|, |u_i^0|^{1/d(i)} ) \Big)$$
so that for all\footnote{This estimate indicates the well known fact that at large scale the Riemannian approximating distances are equivalent to the sub-Riemannian distance}
  $\bar x, \bar x_0\in \bar G$ \begin{equation}
\label{tocheck}
\bar d_0(\bar x, \bar x_0) - C_0\leq \bar d_\e( \bar x, \bar x_0)\leq \bar d_0(\bar x, \bar x_0)+ C_0\end{equation}
where $C_0$ is independent of $\e$. The latter and (\ref{Xgamma}) imply that there is a constant $\tilde C_{s,k}$ independent of $\e$ such that 
$$|(\p_t^s { X}^\e_{i_1}\cdots{ X}^\e_{i_k} \bar \Gamma_{\e, A})(\bar x,\bar y,t)| \le \tilde  C_{s,k} t^{-s-k/2} \frac{e^{-\frac{\bar d_0(\bar x,\bar y)^2}{C_\Lambda t}}} {|\bar B_0(\bar x, \sqrt{t})|}$$
and this imply dominated convergence with respect to the $\e$ variable.
\end{proof}
 
\bigskip

In order to be able to conclude the proof of Proposition \ref{uniform heat kernel estimates}, 
we need to study the relation between the fundamental solutions 
$\Gamma_A( x,y,t)$ and its lifting $\bar \Gamma _{0,A}((x, 0), (y, z), t) $, 
as well as the relation between 
$\Gamma_{\e A} (x,y, t)$and $\bar \Gamma_{\e, A} ((x, 0), (y, z), t) , $

\begin{rmrk}
We first note that for every $f\in C^\infty_0(\R^n \times R^+)$
$f$ can be identified with a $C^\infty$ and bounded function defined on  $\R^{n+ \nu} \times R^+$ and 
constant in the $z-$ variables. Hence 
$$L_{A}f = {\bar L}_{A}f, \quad L_{\e, A}f = {\bar L}_{\e, A}f, \quad$$
Consequently: 
$$f(x,t) = \int \int \Big(\int \bar \Gamma _{\e,A}((x, 0, s), (y, z, t)) dz \Big) L_{\e, A}f(y,s) dy ds $$
 From the definition of fundamental solution 
we  can deduce that   \begin{equation}\label{GAGE}\Gamma_A( x,y,t)=\int_G \bar \Gamma _{0,A}((x, 0), (y, z), t) dz, \quad  \text{ and } \quad\Gamma_{\e A} (x,y, t)=\int_G \bar \Gamma_{\e, A} ((x, 0), (y, z), t) dz, \end{equation}
for any $x\in G$  and $t>0$.

\end{rmrk}

\bigskip

We conclude this section with the proof of the main result Proposition \ref{uniform heat kernel estimates}.

\begin{proof}
In view of the previous remanrk and (global) dominated convergence of  the derivatives of $\bar \Gamma_{\e, A}$ to  the corresponding derivatives of $\bar \Gamma_{0,A}$ as $\e\to 0$, we deduce that  $$\int_G \bar \Gamma_{\e, A} ((x, 0), (y, z), t) dz \rightarrow \int_G \bar \Gamma_{0,A} (( x, 0),(y, z), t) dz$$ as   $\e\rightarrow  0$.
The Gaussian estimates of $ \Gamma_{\e, A}$ follow from the corresponding estimates on $\bar \Gamma_{\e, A}$ and the fact that in view of \eqref{tocheck}, \begin{equation}\label{more-equivalence}\bar d_\e((x, z), (x_0, z_0))\geq \bar d_0((x, z), (x_0, z_0)) - C_0\geq  d_0( x, x_0) +  d_0(z, z_0) - C_0\geq\end{equation}
$$\geq  d_\e(x, x_0) +  d_\e(z, z_0) -3 C_0$$
Indeed the latter shows that there exists a constant $C>0$ depending only on $G, \sigma_0$ such that for every $x\in G$,
$$\int_G e^{- \frac{d^2_\e((x, z), (x_0, z_0))}{t}}dz \leq  C e^{-\frac{d^2_\e(x, x_0)}{t}}
 \int_G e^{- \frac{d^2_\e( z , z_0)}{t}}dz \leq  C e^{-\frac{d^2_\e(x, x_0)}{t}}.$$ 
The conclusion follows at once. 
\end{proof}

\subsection{Differential of the integral operator associated to $\Gamma_\e$}

In this subsection we will show how to differentiate a functional $F$ expressed as follows:
$$F(f)(x, t) = \int \Gamma_{\e, A}(x,y,t) f(y, s) dy ds.$$

In order to do so, we will need to differentiate both with respect to $x$ and to $y$, 
so that we will denote $X_i^{\e, x}\Gamma_{\e, A}(x,y,t)$
the derivative  with respect to the variable $x$ and 
$X_i^{\e, y}\Gamma_{\e, A}(x,y,t)$
the derivative with respect to the variable $y$. 

Analogously, we will denote the derivative with the first variable of the lifted fundamental solution 
$$\bar X_i^{\e, \bar x}\bar \Gamma_{, A}((x, w), (y,z) ,t).$$
For $\e=0$, we will have by definition
$$\bar X_i^{0, \bar x}\bar \Gamma_{, A}((x, w), (y,z) ,t)= ( X_i^{0, x} + \tilde Z_i^w)\bar \Gamma_{, A}((x, w), (y,z) ,t).$$
The derivative with respect to the second variable will be denoted $\bar X_i^{0, \bar y}.$
If $\Gamma$ is the Euclidean heat kernel, there is a simple relation between the derivative with respect to 
the two variables, indeed in this case 
$\Gamma_{\e, A}(x,y,t)= \Gamma_{\e, A}(x-y,0,t)$, so that 
\begin{equation}X_i^{\e, x}\Gamma_{\e, A}(x,y,t)
= - X_i^{\e, y}\Gamma_{\e, A}(x,y,t).\end{equation}
Consequently for every function $f\in C^\infty_0$ 
$$\partial_{x_i}F(f)(x,t) = \int \Gamma_{\e, A}(x,y,t) \partial_{y_i}f(y) dy.$$ 

This is no more the case in general Lie groups, or for H\"ormander vector fields. 
However we will see that there is a relation between the two derivatives, 
which allows to prove the following:%
%
%
%
 %
%
%
\begin{prop}\label{uniformderivatives}
Assume that $f\in C^{\infty}_0(\Omega\times]0, T[)$ in an open set $\Omega\times]0, T[$. 
For every $x\in K\subset\subset \Omega$, for every $i=1, \cdots m$ 
there exists the derivative $X_i^\e F(f)(x,t) $. 
Precisely there exist kernels $ P_{\e, i,h}(x, y, t), R_{\e, i}(x, y, t) 
\in \mathcal{E}(2, d_\e, M^\e_{m, \Lambda}) $ such that 
$$X_i^{\e} F(f) = $$
$$=-\int 
\sum_{h=1}^m  X^{\e, y, *}_{h}  P_{\e, i,h }(x, y, t)    f(y) dy - \int 
R_{\e, i}(x,y, t)   f(y) dy.$$
(Let us note explicitly that the term $R_{\e, i,h}(x,y, t) $ plays the role of 
an error term). 
\end{prop}

\begin{proof}
We can apply the lifting procedure described in sections 5.2 and 5.3, 
and representing the fundamental solution as in  (\ref{changegamma}) and (\ref{GAGE}), 
we obtain the following expression for $F_\e$:
$$F_\e(f) = \int\int_G \bar \Gamma_{\e, A} ((x, 0), (y, z), t) dz f(y) dy =$$
$$= \int \int \bar \Gamma_{0,A}(\bar F_\e(x, 0), \bar F_\e(y, z), t) dz f(y) dy.$$
By differentiating with respect to $X_i^\e$ we get: 
\begin{equation}\label{XFE}X^\e_iF_\e(f)(x) = \int \int ({\bar X^{0,x}}_i - \tilde Z_i^w) \bar \Gamma_{0,A}(\bar F_\e(x, 0), \bar F_\e(y, z), t) dz f(y) dy.\end{equation}
Note that the family of vectors $\bar X^{0}_i$ is independent of $\e$ and 
free of step $r$. Hence, by 
(see \cite{Roth:Stein}, pag 295, line 3 from below) for every $i,j=1, \cdots m$, sure exist families of indices $I_{i,j}$, and polynomials  $ \bar p_{ih}$  homogeneous of degree $\geq h$ such that:
$$\bar X^{0, \bar x}_{i}\bar \Gamma_{0,A}(\bar x, \bar y, t) = $$$$=\sum_{j=1}^m (\bar X^{0, \bar y}_{j})^* \sum_{h\in I_{i,j}}
\bar X^{0, \bar y}_{h} \Big({\bar p_{i h}}(\Theta_{\bar x}(\bar y))\bar \Gamma_{0,A}(\bar x, \bar y, t)\Big)- $$$$-\Big(\sum_{j=1}^m(\bar X^{0, \bar y}_{j})^*\sum_{h\in I_{i,j}}
\bar X^{0, \bar y}_{h}\Big)\Big({\bar p_{ih}}(\Theta_{\bar x}(\bar y))\Big)\bar \Gamma_{0,A}(\bar x, \bar y, t).$$
In particular using this expression in the variable $z$ alone, and integrating by parts we deduce
$$\int \int \tilde Z_i^w \bar \Gamma_{0,A}(\bar F_\e(x, 0), \bar F_\e(y, z), t) dz = 0$$

We now call
$$\bar R_{0, i }(\bar x, \bar y, t)  = \Big(\sum_{j=1}^m(\bar X^{0, y}_{j})^*\sum_{h\in I_{i,j}}
\bar X^{0, \bar y}_{h}\Big)\Big({\bar p_{ih}}(\Theta_{\bar x}(\bar y))\Big)\bar \Gamma_{0,A}(\bar x, \bar y, t)
$$
This kernel, being obtained by multiplication of $\Gamma_{0,A}(\bar x, \bar y, t)$ by a 
polynomial, has locally the same decay as $\Gamma_{0,A}(\bar x, \bar y, t)$. 
In particular it is clear that the conditions \ref{gamma}, \ref{Xgamma}, \ref{AmenA} are satisfied uniformly with respect to $\e$, since there is no dependence on $\e$. 
As a consequence, if we set
$$R_{\e, i }(x, y, t)  =\int \bar R_{0, i }(\bar F_\e(x, 0), \bar F_\e(y, z), t) dz$$
then $R_{\e, i}(x, y, t) \in  \mathcal{E}(2, d_\e, M^\e_{m, \Lambda}) $ 
Similarly we call 
$$\bar P_{\epsilon, i,h }(\bar x, \bar y, t) =
\sum_{h\in I_{i,j}}
\bar X^{0, y}_{h} \Big({\bar p_{i h}}(\Theta_{\bar x}(\bar y))\bar \Gamma_{0,A}(\bar x, \bar y, t)\Big)
$$
Now we use the fact that 
$\bar \Gamma_{0,A}\in \mathcal{E}(2, \bar d, M^\e_{m+\nu, \Lambda}) $ together with the fact that $\bar p_{ih}$ is a polynomial of the degree equal of the order of
$\bar X_{h}^{0, y}$ to conclude that 
$$\bar P_{\epsilon, i,h }(\bar x, \bar y, t) \in \mathcal{E}(2, \bar d, M^\e_{m+\nu, \Lambda}) $$
It follows that, if we call 
$$P_{\e, i,h }(x, y, t)   = \int \bar P_{0, i,h }(\bar F_\e(x, 0), \bar F_\e(y, z), t) dz$$
then $P_{\e, i,h }(x, y, t) \in  \mathcal{E}(2, d_\e, M^\e_{m, \Lambda}) $ 

Plugging this expression into equation (\ref{XFE}) 
we get 
\begin{equation}\label{XFE2}X^\e_iF_\e(f)(x) = 
-\int  
\sum_{j=1}^m  (\bar X^{0, y}_{j})^* P_{\e, i,h }(x, y, t) 
 f(y) dy-\int   R_{\e, i}(x, y, t)   f(y) dy.\end{equation}

\end{proof}

%% file: Schauder.tex
In this section we will prove uniform estimates in spaces of H\"older continuous functions and in Sobolev spaces 
for solutions of second order sub-elliptic differential equations in non divergence form 
$$L_{\e, A} u\equiv \p_t u- \sum_{i,j=1}^n  a^\e_{ij}(x,t) X_i^\e X_j^\e u=0 ,$$
in a cylinder $ Q=\Om\times (0,T)$ that are stable as $\e\to 0$.

Indeed the proof of both estimates is largely based on the knowledge of a fundamental solution. 

Internal Schauder estimates for these type of operators are well known. 
We recall  the results of  Capogna
and Han \cite{CapognaHan} for uniformly subelliptic operators, of Bramanti and Brandolini \cite{BramantiBrandolini}
for heat-type operators,  and the results of Lunardi \cite{Lunardi}, and Polidoro and Di
Francesco \cite{PolidoroDiFrancesco}, and Gutierrez and Lanconelli  \cite{GutierrezLanconelli}, which apply to a large 
class of squares of vector fields plus a drift term.  We also recall \cite{MR2512155} where uniform Schauder estimates for a particular elliptic approximation of subLaplacians are proved.

Here the novelty is due to the uniform condition with respect to $\e$. 
This is accomplished by using the uniform Gaussian bounds established in in the previous section. This result extends to H\"ormander type operators 
the analogous assertion proved by Manfredini and the authors in \cite{CCM3} in the setting of Carnot Groups.


%

\subsection{Uniform Schauder estimates}

Let us start with the definition of classes of H\"older continuous functions in this setting

\begin{dfn}\label{defholder}
Let $0<\alpha < 1$, ${Q}\subset\R^{n+1}$ and  $u$ be defined on
${Q}.$ We say that $u \in C_{\e,X}^{\alpha}({Q})$ if there exists a positive constant $M$ such that for
every $(x,t), (x_{0},t_0)\in {Q}$ 
\begin{equation}\label{e301}
   |u(x,t) - u(x_{0},t_0)| \le M \tilde d_{\e }^\alpha((x,t), (x_{0},t_0)).
\end{equation}
We put  
 $$||u||_{C_{\e,X}^{\alpha}({Q})}=\sup_{(x,t)\neq(x_{0},t_0)} \frac {|u(x,t) - u(x_{0},t_0)|}{\tilde d_{\e }^\alpha((x,t), (x_{0},t_0))}+ \sup_{Q} |u|.$$
 Iterating this definition, 
 if $k\geq 1$  we say that $u \in
C_{\e,X}^{k,\alpha}({Q})$  if for all $i=1,\ldots ,m$
 $X_i \in C_{\e,X}^{k-1,\alpha}({Q})$.
 Where we have set $C^{0,\alpha}_{\e,X}({Q})=C^{\alpha}_{\e, X}({Q}).$ 
\end{dfn}

The main results of this section, which generalizes to the H\"ormander vector fields setting 
our previous result with Manfredini  in \cite{CCM3} is

\begin{prop}\label{ellek}
Let  $w$ be a smooth  solution of $L_{\e, A}w=f$ on ${Q}$.
Let $K$ be a compact sets such that  $K\subset\subset {Q}$,  set $2\delta=d_0(K, \p_p Q)$ and
denote by $K_\delta$ the $\delta-$tubular neighborhood of $K$. Assume that 
there exists a constant $C>0$ such that 
$$ || a_{ij}^\e||_{C^{k,\alpha}_{\e,X}(K_\delta)} \leq C,$$ for any $\e\in (0,1)$.
There exists a constant $C_1>0$ depending  on
$\alpha$, $C$, $\delta$, and the constants in Proposition \ref{uniform heat kernel estimates},
but independent of $\e$,  such that
$$||w||_{C^{k+2, \alpha}_{\e,X}(K)} \leq C_1 \left( ||f||_{C^{k,\alpha}_{\e,X}(K_\delta)}+ ||w||_{C^{k+1, \alpha}_{\e,X}(K_\delta)}\right). $$
\end{prop}

We will first consider to a constant coefficient operator, for which we will 
obtain a representation formula, then we will show how to obtain from this 
the claimed result. 

Precisely we will consider the constant coefficient  {\it  frozen } operator:
$$L_{\e,(x_0,t_0)}\equiv \p_t - \sum_{i,j=1}^n a^\e_{ij}(x_0,t_0) X_i^\e X_j^\e,$$
where $(x_0,t_0)\in Q$. We explicitly note that for $\epsilon>0$ fixed the operator $L_{\e, (x_0,t_0)}$ is uniformly parabolic, so that 
its heat kernel can be studied through standard singular integrals theory in the corresponding Riemannian balls.

As a direct consequence of the definition of fundamental solution one has the following representation formula
\begin{lemma}\label{representation} Let $w$ be a smooth solution to
$L_\e w=f $ in $Q\subset\R^{n+1}$.
For every  $\phi \in C^\infty_0(Q)$,
\begin{align}\label{repres}
&   (w\phi)(x,t)\\ \nonumber 
& =   \int_Q\G^\e_{(x_0,t_0)}((x,t),(y,\tau))
\Big(L_{\e, (x_0,t_0)}-L_\e\Big)(w \,\phi)(y,\tau) dyd\tau+\\ \nonumber 
&+\int_{Q}\G^\e_{(x_0,t_0)}((x,t),(y,\tau))
\Big( f\phi  +  wL_\e\phi + 2 \sum_{i,j=1}^n a_{ij}^\e (y,\tau) X_i^\e wX_j^\e \phi\Big)(y,\tau) dyd\tau,
\end{align}
where  we have denoted by $\G^\e_{(x_0,t_0)}$ the heat kernel for  of $L_{\e,(x_0,t_0)}$.
\end{lemma}

Iterating the previous lemma we get the following
\begin{lemma}\label{kderiv} Let $k\in N$ and consider a  $k-$tuple $(i_1,\ldots,i_k)\in \{1,\ldots,m\}^k$.
There exists a differential operator  $B$ of order $k+1$,   depending on horizontal derivatives of $a_{ij}^\e$ of order at most $k$, such that  
$$ X_{i_k}^\e\cdots X_{i_1}^\e  \big(L_{\e,(x_0,t_0)}-L_\e\big)= \sum_{i,j=1}^n \Big(
a_{ij}^\e - a_{ij}^\e(x_0, t_0) \Big)  X_{i_k}^\e\cdots X_{i_1}^\e X^\e_iX^\e_j  + B. $$
\end{lemma}
\begin{proof}
The proof can be made by induction. Indeed it is true with $B=0$ by definition if $k=0$:
$$L_{\e,(x_0,t_0)}-L_\e= \sum_{i,j=1}^n \Big(a_{ij}^\e - a_{ij}^\e(x_0, t_0) \Big)   X^\e_iX^\e_j . $$
if it true for a fixed value of $k$ then we have 
$$ X_{i_{k+1}}^\e X_{i_1}^\e\cdots X_{i_k}^\e  \big(L_{\e,(x_0,t_0)}-L_\e\big)= \sum_{i,j=1}^n \Big(
a_{ij}^\e - a_{ij}^\e(x_0, t_0) \Big)  X_{i_{k+1}}^\e X_{i_k}^\e\cdots X_{i_1}^\e X^\e_i X^\e_j  + \tilde B
 $$
where
$$\tilde B=X_{i_{k+1}}(a_{ij}^\e - a_{ij}^\e(x_0, t_0) )  X_{i_k}^\e\cdots X_{i_1}^\e X^\e_iX^\e_j  
+ X_{i_{k+1}}B.$$
By the properties of $B$ it follows that $\tilde B$ is a differential operator of order $k+2$,   depending on horizontal derivatives of $a_{ij}^\e$ of order at most $k+1$. This concludes the proof.
\end{proof}

\bigskip

We can go back to our operator $L$ and establish the following regularity results, 
differentiating twice the representation formula:

\begin{prop}\label{elle2} 
Let  $0<\alpha<1$  and  $w$ be a smooth  solution of $L_{\e}w=f\in C^{\alpha}_{\e.X}({Q})$ in the cylinder   ${Q}$.
Let $K$ be a compact sets such that  $K\subset\subset {Q}$,  set $2\delta=d_0(K, \p_p Q)$ and
denote by $K_\delta$ the $\delta-$tubular neighborhood of $K$.
 Assume that 
there exists a constant $C>0$ such that for every $\e\in (0,1)$
$$ || a^\e_{ij}||_{C^{ \alpha}_{\e,X}(K_\delta)}\leq C. $$
There exists a constant $C_1>0$ depending  on $\delta$,
$\alpha$, $C$ and the constants in Proposition \ref{uniform heat kernel estimates}
such that
$$||w||_{C^{2, \alpha}_{\e,X}(K)} \leq  C_1\left( ||f||_{C^{\alpha}_{\e,X}(K_\delta)}+ ||w||_{C^{1, \alpha}_{\e,X}(K_\delta)}\right).$$
\end{prop}

\begin{proof} The proof follows the outline of the standard case, as in  \cite{Friedman}, and rests crucially on the Gaussian estimates proved in  Proposition \ref{uniform heat kernel estimates}. 
Choose a parabolic sphere\footnote{That is a sphere in the group $\tilde G=G\times \R$ in the pseudo-metric $\tilde{d}_\e$ defined in \eqref{defde}.} $B_{\e,\delta} \subset\subset K$ where  $\delta>0$ will be fixed later and 
a cut-off function $\phi\in C^\infty_0(\R^{n+1})$ identically 1 on $B_{\e, \delta/2}$ and compactly supported in $B_{\e,\delta}$. This  implies that for some constant $C>0$ depending only on $G$ and $\s_0$, 
$$\left|\nabla_\e \phi\right|\leq C\delta^{-1}, \quad |L^\e \phi| \leq C\delta^{-2},$$ in $Q.$
Now we represent the function $w\phi$ through the formula \ref{representation} 
and take two derivatives in the direction of the vector fields. We remark once more that the operator is 
uniformly elliptic due to the 
$\e-$regularization, hence the differentiation under the integral can ben 
considered standard. 
As a consequence
for every multi-index $I=(i_1,i_2)\in \{1,\ldots ,m\}^2$  and   for every $(x_0,t_0)\in B_{\e,\delta}$
one has:
 \begin{align}\label{represDeriv}
& X_{i_1}^\e X_{i_2}^\e  (w\phi)(x_0,t_0)  \\ = &  \int_Q X_{i_1}^\e X_{i_2}^\e \G^\e_{(x_0,t_0)}(\cdot,(y,\tau))|_{(x_0,t_0)}
\left(L_{\e,(x_0,t_0)}-L_\e\right)(w \,\phi)(y,\tau) dyd\tau+\\ \nonumber 
    & + \int_Q X_{i_1}^\e X_{i_2}^\e \G^\e_{(x_0,t_0)}(\cdot,(y,\tau))_{(x_0,t_0)}
\left( f\phi  +  wL_\e \phi + 2\sum_{i,j=1}^n  a_{ij}^\e X_i^\e wX_j^\e \phi\right)(y,\tau) dyd\tau.
\end{align} 
In order to study the H\"older continuouity of the second derivatives, we note that the uniform  H\"older continuity of $a_{ij}^\e$, and Proposition \ref{uniform heat kernel estimates}
ensure that the kernal satisfy the classical singular integral properties (see \cite{fol:1975}): 
\begin{align*}|X_{i_1}^\e X_{i_2}^\e \G^\e_{(x,t)}((x,t),(y,\tau))&- X_{i_1}^\e X_{i_2}^\e \G^\e_{(x_0,t_0)}((x_0,t_0),(y,\tau))|\\
&\le C\,\tilde d^\alpha _\e((x,t),(x_0,t_0)) 
\frac{(\tau-t_0)^{-1}e^{-\frac{d_\e(x_0,y)^2}{C_\Lambda (\tau-t_0)}}} {|B_\e(0, \sqrt{\tau-t_0})|},
\end{align*}
with $C>0$ independent of $\e$.  
From here, proceeding as in \cite[Theorem 2, Chapter 4]{Friedman}, 
the first term in the right hand side of formula (\ref{represDeriv})
 can be  estimated as follows:
\begin{align}\label{estimate1}
\Big|\Big|
\int X_{i_1}^\e X_{i_2}^\e\G^\e_{(x_0,t_0)} &(\cdot,(y,\tau)) (L_\e-L_{\e,(x_0, t_0)})
(w \,\phi )(y,\tau)dyd\tau \Big|\Big|_{C^\alpha_{\e,X}(B_{\e,\delta})} \\ \nonumber
&\le C_1 \Big|\Big|(L_\e-L_{\e,(x_0, t_0)})
(w \,\phi )\Big|\Big|_{C^\alpha_{\e,X}(B_{\e, \delta})} \\ \nonumber
&= C_1\sum_{i,j}\left|\left|(a_{ij}^\e(x_0,t_0)-a_{ij}^\e(\cdot)\big)X_j^\e X_j^\e(w \,\phi)
   \right|\right|_{C^\alpha_{\e,X}(B_{\e, \delta})}\\ \nonumber
&\le \tilde C_1 \delta^\alpha  ||a_{ij}^\e||_{C^{\alpha}_{\e,X}(B_{\e,\delta})}||w\phi||_{C^{2, \alpha}_{\e,X}(B_{\e,\delta})},
\end{align}
where $C_1,$ and $\tilde C_1$ are  stable  as $\e\to 0$.
Similarly, if $\phi$ is fixed, the H\"older norm of the second term in the representation formula (\ref{represDeriv}) is bounded by
\begin{align}\label{estimate2}
\Big|\Big|\int X_{i_1}^\e X_{i_2}^\e\G^\e_{(x_{0},t_0)}((x_0,t_0),(y,\tau))
&\big(f\phi(y,\tau) + wL\phi(y,\tau) + 2 a_{ij}^\e X_i^\e wX_j^\e \phi \big)dyd\tau\Big|\Big|_{C^\alpha_{\e,X}(B_{\e, \delta})}\\  \nonumber
& \leq C_2 \left(||f||_{C^\alpha_{\e,X}(K_\delta)} +\frac{C}{\delta^2}|| w||_{C^{1, \alpha}_{\e,X}(K_\delta)}\right).
\end{align}
From  (\ref{represDeriv}), (\ref{estimate1}) and (\ref{estimate2})  we deduce that 
$$||w\phi||_{C^{2, \alpha}_{\e,X}(B_\delta)}\leq  \tilde C_2\, \delta^\alpha ||w\phi||_{C^{2, \alpha}_{\e,X}(B_\delta)} + C_2\left(||f||_{C^\alpha_{\e,X}(K_\delta)} +\frac{C}{\delta^2}|| w||_{C^{1,\alpha}_{\e,X}(K_\delta)}\right).$$
Choosing $\delta$ sufficiently small we prove  the assertion  on the fixed sphere $B_{\e,\delta}$
The conclusion follows from  a standard  covering argument.

\end{proof}

We can now conclude the proof of proposition \ref{ellek}: 

\begin{proof}
The proof is similar to the  previous one for $k=1$. We start by differentiating 
  the  representation formula (\ref {repres}) along 
		an arbitrary direction $X_{i_1}$
		 \begin{align}\label{represderivpparti}
& X_{i_1}^\e   (w\phi)(x,t)  \\ = &  \int_Q X_{i_1}^\e  \G^\e_{(x_0,t_0)}(\cdot,(y,\tau))
\left(L_{\e,(x_0,t_0)}-L_\e\right)(w \,\phi)(y,\tau) dyd\tau+\\ \nonumber 
    & + \int_Q X_{i_1}^\e  \G^\e_{(x_0,t_0)}(\cdot,(y,\tau))
\left( f\phi  +  wL_\e \phi + 2\sum_{i,j=1}^n  a_{ij}^\e X_i^\e wX_j^\e \phi\right)(y,\tau) dyd\tau.
\end{align} 
Now we apply Theorem \ref{uniformderivatives} and deduce that there exist kernels
$$P_{e, i_1,h, (x_0,t_0)}((x,t),(y,\tau)), 
R_{e, i_1, (x_0,t_0)}((x,t),(y,\tau)), $$

 with the same decay of the fundamental solution 
such that 

		 \begin{equation}\label{represderivpparti}
 X_{i_1}^\e   (w\phi)(x,t) =
\end{equation} 

	$$=-\int 
\sum_{h=1}^m  P_{\e, i_1,h, (x_0,t_0)}((x,t),(y,\tau))   X^{\e, y}_{h}\left(L_{\e,(x_0,t_0)}-L_\e\right)(w \,\phi)(y,\tau) dyd\tau -$$
$$-\int R_{\e, i_1,(x_0,t_0)}((x,t),(y,\tau)) \left(L_{\e,(x_0,t_0)}-L_\e\right)(w \,\phi)(y,\tau) dyd\tau dy-$$
	$$-\sum_{h=1}^m\int P_{\e, i_1,h, (x_0,t_0)}((x,t),(y,\tau))   X^{\e, y}_{h}\left( f\phi  +  wL_\e \phi + 2\sum_{i,j=1}^n  a_{ij}^\e X_i^\e wX_j^\e \phi\right)(y,\tau) dyd\tau -$$
$$-\int 
R_{\e, i_1, (x_0,t_0)}((x,t),(y,\tau))\left( f\phi  + 
 wL_\e \phi + 2\sum_{i,j=1}^n  a_{ij}^\e X_i^\e wX_j^\e \phi\right)(y,\tau) dyd\tau .$$

Using Lemma \ref{kderiv}, this yields the existence of new kernels 
$  P^{i_1,\cdots, i_k}_{\e, h_1, \cdots, h_k, (x_0,t_0)}((x,t),(y,\tau)) $ with the behavior of a fundamental solution 
(and the same dependence on $\epsilon$) such that 

\begin{align*}
& X_{i_1}^\e\cdots X_{i_{k}}^\e    (w\phi)(x,t)\\&
 =   \int   ^{i_1,\cdots, i_k}_{\e, h_1, \cdots, h_k, (x_0,t_0)}((x,t),(y,\tau))
\Big(
a_{ij}^\e - a_{ij}^\e(x_0, t_0) \Big)X_{i_1}^\e\cdots X_{i_{k}}^\e    X^\e_iX^\e_j (w \,\phi)(y,\tau) dyd\tau\\
&+  \int ^{i_1,\cdots, i_k}_{\e, h_1, \cdots, h_k, (x_0,t_0)}((x,t),(y,\tau))B (w \,\phi)(y,\tau) dyd\tau \\
&+\int ^{i_1,\cdots, i_k}_{\e, h_1, \cdots, h_k, (x_0,t_0)}((x,t),(y,\tau))
 X_{i_1}^\e\cdots X_{i_{k}}^\e  \Big(f\phi(y,\tau) + wL_\e\phi(y,\tau) + 2 \sum_{i,j=1}^n a_{ij}^\e X_i^\e wX_j^\e\phi \Big)dyd\tau+\\
&+ {\text {lower order terms}}\\
\end{align*}
where $\phi$ is as in the proof of Proposition \ref{elle2} and $B$ is a differential operator of order $k+1$.
The conclusion follows by further differentiating the previous representation formula along two horizontal directions $X_{j_1}^\e X_{j_2}^\e$ and arguing  as in the proof of  Proposition \ref{elle2}. 
\end{proof}

%

%% file: Harnack.tex
The results we have presented so far show that for any $\e_0>0$,  the $1-$parameter family of metric spaces $(\M,d_\e)$ associated to the Riemannian approximations of a subRiemannian metric space $(\M,d_0)$, satisfy uniformly in $\e\in [0,\e_0]$ the hypothesis in the definition  of {\it $p$-admissible structure} in the sense of \cite[Theorem 13.1]{hak:sobolev2}. This class of metric measure spaces has a very rich analytic structure (Sobolev-Poincar\'e inequalities, John-Nirenberg lemma, ...) that allows for the development of a basic regularity theory for weak solutions of classes of  nonlinear degenerate parabolic and elliptic PDE. In the following we will remind the reader of the definition and basic properties of $p-$admissible structures and sketch some of the main regularity results from the recent papers \cite{ACCN} and \cite{CCR}. We will conclude the section with a sample application of these techniques to the global (in time) existence of solutions for a class of evolution equations that include the subRiemannian total variation flow \cite{CCM3}.

 \subsection{Admissible ambient space geometry}
Consider a smooth real manifold $M$ and  let $\mu$ be a locally finite Borel measure on $M$ which is absolutely continuous with respect the Lebesgue measure when represented in local charts. Let   $d(\cdot, \cdot): M \times M \to \R^{+}$ denote  the  control distance generated by a system of bounded, $\mu$-measurable, Lipschitz vector fields $\X = (X_1,\ldots, X_m)$ on $M$. As in \cite{CCgeometry:Risler} and \cite{GarofaloNhieu:lip1998} one needs to assume as basic  hypothesis 
 \begin{equation}\label{topology}
	\text{ the inclusion }	i: (M, chart) \to (M,d) \text{ is continuous,}
	\end{equation}
	where we have denoted by $(M, chart)$ the topology on $M$ induced by the Euclidean topology in $\R^n$ via coordinate charts.
	%
	%
		For $x \in M$ and $r > 0$,  set $B(x,r) = \{y \in M: d(x,y) < r \}$ and let  $|B(x,r)|$ denote the $\mu$ measure of $B(x,r)$. In general, given a function $u$ and a ball $B=B(x,r)$ then $u_B$ denotes the $\mu$-average of $u$ on the ball $B=B(x,r)$. In view of \eqref{topology} the closed metric ball $\bar B$ is a compact set.

\begin{dfn}\label{admissible} Assume hypothesis \eqref{topology} holds. Given $1 \leq p < \infty$, the triplet $(M,\mu,d)$ is said to define a $p$-admissible structure (in the sense of \cite[Theorem 13.1]{hak:sobolev2}) if for every compact subset $K$ of $M$ there exist constants $C_D = C_D(\X, K), C_P = C_P(\X, K) > 0$, and $R = R(\X, K) > 0$, such that the following hold.
\begin{enumerate}
\item {\it Doubling property:}\begin{equation} \tag{D} \label{eq_D}
	 |B(x,2r)|\leq C_D |B(x,r)|\mbox{ whenever $x \in K$ and $0 < r < R$}.
\end{equation}
\item {\it Weak $(1,p)$-Poincar\'e inequality:}
\begin{equation}\tag{P} \label{eq_P}
	\fint_{B(x,r)} |u - u_B| d\mu \leq C_P \, r \left ( \fint_{B(x,2r)} |\X u|^p d\mu \right )^{1/p},
\end{equation}
whenever $x \in K$, $0 < r < R$,
$u\in W_{\X}^{1,p}(B(x,2r)).$

\end{enumerate}
\end{dfn}

Theorems \ref{Main-1} and \ref{Poincare-epsilon}
yield the following

\begin{thrm}\label{after all this work} Let $X_1,...,X_m$ be a family of H\"ormander vector fields in $\Om\subset \R^n$ and denote by $\mu$ Lebesgue measure. For each $\e_0>0$ and $\e\in [0,\e_0]$ denote by $d_\e$ the distance functions defined in Definition \ref{prima-def}.
For all $\e\in [0,\e_0]$ and $p\ge 1$,  the space $(\Om, \mu, d_\e)$ is $p-$admissible, with constants $C_D$ and $C_P$ independent of $\e$.
\end{thrm}

Other examples of $p-$admissible spaces are:
\begin{itemize}
	 \item The classical setting: $\M=\mathbb R^n$, $d\mu$ equals the $n$-dimensional Lebesgue measure,  and $\X = (X_1,\ldots, X_m)=(\partial_{x_1},...,\partial_{x_n})$.

%
\item 
Our setting is also sufficiently broad to include some non-smooth vector fields such as the Baouendi-Grushin frames, e.g., consider, for $\gamma\ge 1$ and $(x,y)\in \R^2$, the vector fields $X_1=\partial_x$ and $X_2=|x|^{\gamma}\partial_y$. Unless $\gamma$ is a positive even integer these vector fields fail to satisfy H\"ormander's finite rank hypothesis. However, the doubling inequality as well as the Poincar\'e inequality hold and have been used in the work of Franchi and Lanconelli \cite{MR753153} to establish Harnack inequalities for linear equations.

\item Consider a smooth manifold $\M$ endowed with a complete Riemannian metric $g$. Let $\mu$ denote the Riemann volume measure, and by $\X$ denote a $g-$orthonormal frame. If the Ricci curvature is bounded from below ($Ricci\ge -Kg$) then one has a $2-$admissible structure. In fact, in this setting the Poincar\'e inequality follows from Buser's inequality while the doubling condition is a consequence of the Bishop-Gromov comparison principle. If $K=0$, i.e. the Ricci tensor is non-negative, then these assumptions holds globally and so does the Harnack inequality.

\end{itemize}
 Spaces with a $p$-admissible structure support a homogenous structure in the sense of Coifman and Weiss \cite{cw:1971}.  
\begin{lemma} \label{lem_2.1} Let  $(\M,\mu,d)$ be a $p$-admissible structure for some $p\ge 1$,  $\Omega$  a bounded open set in $\M$ and set $K=\bar\Omega$. 
If $x \in K$ and $0 < s < r < R$, then the following holds.
	\begin{enumerate}
		\item There exists a constant $N = N(C_D) > 0$, called homogeneous dimension of $K$ with respect to $(\X, d, \mu)$, such that $|B(x,r)| \leq C_D \tau^{-N} |B(x,\tau r)|$, for all $0 < \tau \leq 1$.
		\item There exists a continuous function $\phi \in C_{0}(B(x,r))\cap W_{\X}^{1,\infty}(B(x,r))$ and a constant $C = C(\X,K) > 0$, such that $\phi = 1$ in $B(x,s)$ and $|\X \phi|\leq C/(r-s)$, $0 \leq \phi \leq 1$.
		\item Metric balls have the so called $\hat \delta-$annular decay property, i.e., there exists $\hat \delta = \hat \delta(C_D) \in (0,1]$, such that
		\begin{equation*}
			|B(x,r)\setminus B(x,(1-\epsilon)r)| \leq C \epsilon^{\hat \delta} |B(x,r)|,
		\end{equation*}
whenever $0 < \epsilon < 1$.
	\end{enumerate}
\end{lemma}

\begin{proof}
	Statement (1) follows from \eqref{eq_D} by a standard iteration argument. Statement (2) is proved in \cite[Theorem 1.5]{GarofaloNhieu:lip1998}.
Statement 	(3) follows from \cite[Corollary 2.2]{Bu}, since we have a \CC space. Furthermore, $\hat \delta$ depends only on $C_D$.
\end{proof}

Given $\Omega \subset M$, open, and $1 \leq p \leq \infty$, we let $W_{\X}^{1,p}(\Omega) = \{ u \in L^p(\Omega,\mu): X_i u \in L^p (\Omega,\mu), i=1,...,m\}$ denote the {\it horizontal} Sobolev space, and we let $W_{\X,0}^{1,p}\subset W_{\X}^{1,p}$ be the closure\footnote{Here we avoid the issue ``$H=W$''. This is studied in detail in \cite{Friedrichs}, \cite{Kilpelainen}, \cite{GarofaloNhieu:lip1998}, \cite{FSS1}, \cite{frss:embedding} and \cite{Serra}.} of the space of $W_{\X}^{1,p}$ functions with compact (distributional) support in the norm $\| u \|_{1,p}^p = \| u \|_p + \| \X u \|_p$ with respect to $\mu$. In the following we will omit $\mu$ in the notation for Lebesgue and Sobolev spaces.
Given
$t_1<t_2$, and $1 \leq p \leq \infty$, we let $\Omega_{t_1,t_2} \equiv \Omega \times (t_1,t_2)$ and we let $L^p(t_1,t_2;W_{\X}^{1,p}(\Omega))$, $t_1 < t_2$, denote the parabolic Sobolev space of real-valued functions defined on $\Omega_{t_1,t_2}$ such that for almost every $t$, $t_1 < t < t_2$, the function $x \to u(x,t)$ belongs to $W_{\X}^{1,p}(\Omega)$ and
\begin{equation*}
	\|u \|_{L^{p}(t_1,t_2;W_{\X}^{1,p}(\Omega))} =
	\left ( \int_{t_1}^{t_2} \int_{\Omega} (|u(x,t)|^p + |\X u(x,t)|^p) d\mu dt \right )^{1/p} < \infty.
\end{equation*}
The spaces $L^p(t_1,t_2;W_{\X,0}^{1,p}(\Omega))$ is defined analogously. We let $W^{1,p}(t_1,t_2;L^p(\Omega))$ consist of real-valued functions $\eta \in L^p(t_1,t_2;L^p(\Omega))$ such that the weak derivative $\partial_t\eta(x,t)$ exists and belongs to $L^p(t_1,t_2;L^p(\Omega))$. Consider the set of functions $\phi$, $\phi \in W^{1,p}(t_1,t_2;L^p(\Omega))$, such that the functions
\begin{equation*}
	t\to \int_{\Omega} |\phi(x,t)|^p d \mu(x) \mbox{ and } t\to \int_{\Omega} |\partial_t \phi(x,t)|^p d \mu(x),
\end{equation*}
 have compact support in $(t_1,t_2)$. We let $W_0^{1,p}(t_1,t_2; L^p(\Omega))$ denote the closure of this space under the norm in $W^{1,p}(t_1,t_2; L^p(\Omega))$.

 From  \cite[Corollary 9.5]{hak:sobolev2} one can see that the metric balls  $B(x_0,r)$ are John domains. Consequently,   \eqref{eq_D}, \eqref{eq_P}, and \cite[Theorem 9.7]{hak:sobolev2} yield  Sobolev-Poincar\'e inequality,  
\begin{lemma} \label{lem_sobolev}
	Let $B(x_0,r) \subset \Omega$, $0 < r < R$, $1 \leq p<\infty$. There exists a constant $C = C(C_D,C_P,p) \geq 1$ such that for every  $u \in W_{\X}^{1,p} (B(x_0,r))$,
	\begin{equation*}
		\left ( \fint_{B(x_0,r)} |u-u_B|^{\kappa p} d\mu \right )^{1/\kappa} \leq C r^p \fint_{B(x_0,r)} |\X u|^p d\mu,
	\end{equation*}
	where $u_B$ denotes the $\mu$ average of $u$ over $B(x_0,r)$,  and where $1 \leq \kappa \leq {N}/{(N-p)}$, if $1 \leq p < N$, and $1 \leq \kappa < \infty$, if $p \geq N$. Moreover, 
	\begin{equation*}
		\left ( \fint_{B(x_0,r)} |u|^{\kappa p} d\mu \right )^{1/\kappa} \leq C r^p \fint_{B(x_0,r)} |\X u|^p d\mu,
	\end{equation*}
	whenever $u \in W_{\X,0}^{1,p} (B(x_0,r))$.
\end{lemma}

\subsection{Quasilinear degenerate parabolic PDE} In this section we list some recent results concerning regularity of weak solutions of certain nonlinear, degenerate parabolic PDE in  spaces \ $(M,\mu,d)$ that are  $p$-admissible  for some $p\in [2,\infty)$. If $p=2$ we can allow lower order terms, but at present this is not yet established for $p>2$. Given a domain (i.e., an open, connected set) $\Omega\subset M$, and $T>0$ we set $\Omega_T=\Omega\times (0,T)$.   For a function $u:\Om_T\to \R$, and $1\le p,q$ we define the norms
\begin{equation}\label{lpq-norms}
||u||_{p,q}^q=\Big(\int_0^T (\int_\Om |u|^p dx)^{\frac{q}{p}} dt\Big)^{\frac{1}{q}},
\end{equation}
and the corresponding Lebesgue spaces $L^{p,q}(\Om_T)=L^q([0,T], L^p(\Om))$. We will say that $\A, \B$ are {\it admissible symbols} (in $\Omega_T$) if the following holds:
\begin{enumerate}[(i)]
	\item  
	 $(x,t) \to \A(x,t,u,\xi), \B(x,t,u,\xi)$ are measurable for every $(u,\xi) \in \R \times \R^m$,
	 \item $(u,\xi) \to \A(x,t,u,\xi), \B(x,t,u,\xi)$ are continuous for almost every $(x,t) \in \Omega_T$,
	 \item 
	  \begin{itemize}
	  \item For $p=2$: There exist constants $a,\bar a>0$ and functions
$b,c,e,f,h\in L^{p,q}(Q)$ with $p>2$, and $q$ given by  $\frac{N}{2p}+\frac{1}{q}<\frac{1}{2}$ and functions  $d,g\in L^{\al,\beta}(Q)$ with $1<\al$ and $\beta$ given by $\frac{N}{2\al}+\frac{1}{\beta}<1$ such that for a.e. $(x,t) \in \Omega_T$ and $\xi\in \R^m$ one has
\begin{eqnarray}\label{structure}
\sum_{i=1}^{{m}} \A_i (x,t,u,\xi) \xi_i && \ge a|\xi|^2-b^2 u^2-f^2, \notag\\
|\A(x,t,u,\xi)| && \le \bar a |\xi |+e|u|+h,\\
|\B(x,t,u,\xi)| && \le c|\xi|+d|u|+g. \notag
\end{eqnarray}
In view of the conditions on $p,q,\al,\beta$ there exists  $\theta>0$ such that
\begin{eqnarray}\label{theta}
p\ge \frac{2}{1-\theta} && \text{ and } \frac{{N}}{2p}+\frac{1}{q}\le \frac{1-\theta}{2} \notag \\
\al \ge \frac{1}{1-\theta}  &&\text{ and } \frac{{N}}{2\alpha}+\frac{1}{\beta} \le 1-\theta.
\end{eqnarray}
We say that {\it a constant depends on the structure conditions \eqref{structure}}, if it depends only on \footnote{ The $||\cdot||$ norms are in the appropriate  $L^{p,q}$ or $L^{\al, \beta}$ classes}
$$a,\bar a, ||b||, ||c||,||d|| ,||e||, ||f||,||g||,||h||, N, \theta,$$ and is uniformly bounded if these quantities are so.
	
	  \item For $p>2$ we will only consider $\B=0$ and ask that  the following bounds 
	\begin{equation}\label{admissiblesym}
		 \A(x,t,u,\xi) \cdot \xi \geq \A_0 |\xi|^p, \ |\A(x,t,u,\xi)| \leq \A_1 |\xi|^{p-1},
	\end{equation}
	hold for every $(u,\xi) \in \R \times \R^m$ and almost every $(x,t) \in \Omega_T$.
	\end{itemize}
\end{enumerate}
 $\A_0$ and $\A_1$ are called the structural constants of $\A$. If $\A$ and $\tilde \A$ are both admissible symbols, with the same structural constants $\A_0$ and $\A_1$, then we say that the symbols are structurally similar.

Let $E$ be a domain in $\M \times \R$. We say that the function $u:E\to\R$ is a weak solution to
\begin{equation} \label{eq_theeq}
\partial_t u(x,t)=	L_{A,p} u \equiv -\sum_{i=1}^{m}X_i^\ast \A_i(x,t,u,\X u)+ \B(x,t,u,\X u),
\end{equation}
 in $E$, where $X^\ast_i$ is the formal adjoint w.r.t. $d \mu$, if whenever $\Omega_{t_1,t_2} \Subset E$ for some domain $\Omega\subset\M$, $u \in L^p(t_1,t_2;W_{\X}^{1,p}(\Omega))$ and
	\begin{equation} \label{eq_sol}
		\int_{t_1}^{t_2} \int_{\Omega} u \frac{\partial \eta}{\partial t} d\mu dt - \int_{t_1}^{t_2} \int_{\Omega} \A(x,t,u,\X u) \cdot \X \eta\ d\mu dt + \int_{t_1}^{t_2} \int_{\Omega} \B(x,t,u,\X u)  \eta\ d\mu dt  = 0,
	\end{equation}
	for every  test function $$\eta \in W_0^{1,2}(t_1,t_2; L^2(\Omega)) \cap L^p (t_1,t_2; W_{\X,0}^{1,p}(\Omega)). $$ A function $u$ is a weak super-solution (sub-solution) to \eqref{eq_theeq} in $E$ if
whenever $\Omega_{t_1,t_2} \Subset E$ for some domain $\Omega\subset\M$, we have $u \in L^p(t_1,t_2;W^{1,p}(\Omega))$, and the left hand side of \eqref{eq_sol} is non-negative (non-positive) for all non-negative test functions $W_0^{1,2}(t_1,t_2; L^2(\Omega)) \cap L^p (t_1,t_2; W_{\X,0}^{1,p}(\Omega))$.

The main results in \cite{ACCN} and \cite{CCR} can be summarized in the following theorem.

\begin{thrm}\label{main-th} Let $(\M,\mu,d)$ be a $p$-admissible structure for some fixed $p\in [2,\infty)$. For a bounded open subset $\Omega\subset \M$, let $u$ be a non-negative, weak solution to \eqref{eq_theeq} in an open set containing the cylinder $\Omega\times [0,T_0]$ and assume that the structure conditions \eqref{admissiblesym} are satisfied.

\begin{itemize}  \item For $p=2$ and for any  subcylinder $Q_{3\rho}=B(\bar x,3\rho)\times(\bar t-9\rho^2,\bar t)\subset Q$  there exists a constant $C>0$ depending on
$C_D$, $C_L$, $C_P$, the structure conditions \eqref{structure} and on $\rho$ such that
\begin{equation}
\sup_{Q^-}u\le C\inf_{Q^+} (u+\rho^\theta k),
\end{equation}
where \begin{equation}\label{cyl-1}Q^+=B(x,\rho)\times(\bar t-\rho^2,\bar t)\text{  and }Q^-=B(x,\rho)\times(\bar t- 8\rho^2, \bar t-7 \rho^2)\end{equation}
$\theta>0$ is defined as in \eqref{theta}, and we have let $k=||f||+||g||+||h||$.

\item For $p>2$: Assuming $\B=0$, there exist constants $C_1,C_2,C_3 \geq 1$, depending only on $\data$, such that for almost all $(x_0,t_0)\in \Omega\times [0,T_0]$, the following holds: If $u(x_0,t_0)>0$, and if $0<r\le R(\X, \bar \Omega)$ (from Definition \ref{admissible}) is sufficiently small so that
\begin{equation*}
	B(x_0,8r)\subset \Omega \quad \text{ and }\quad (t_0 - C_1 u(x_0,t_0)^{2-p}{r}^p,\
t_0 + C_1 u(x_0,t_0)^{2-p}{r}^p) \subset (0,T_0),
\end{equation*}
then
\begin{equation*}
	u(x_0, t_0) \le C_2\inf_{Q}u,
\end{equation*}
where
\begin{equation*}
	Q=B(x_0, {r}) \times \bigg(t_0 +\frac 1 2{C_3} { u(x_0,t_0)^{2-p}{r}^p},
t_0 + C_3 u(x_0,t_0)^{2-p}{r}^p\bigg).
\end{equation*}
Furthermore, the constants $C_1,C_2,C_3$ can be chosen independently of $p$ as $p\to 2$.
\end{itemize}
\end{thrm}

We conclude this section with a 
corollary of the proof in \cite{CCR}[Lemma 3.6], a weak Harnack inequality that plays an important role in the proof of the regularity of the mean curvature flow for graphs over certain Lie groups established in \cite{CCM4}.
Consider a weak supersolution $w \in L^p(t_1,t_2;W_{\X}^{1,p}(\Omega))$ of the linear equation
\begin{equation}\label{linear}
-\p_t w- \sum_{i=1}^m X_i^* (a_{ij}(x,t) X_j w)= g(x,t),
\end{equation}
with $t_1,t_2,\Om$ as defined above. Assume the coercivity hypothesis 
\begin{equation} \label{unifellip-l}\Lambda^{-1} \sum_{d(i)=1} \xi_i^2 
 \leq \sum_{i,j=1}^m a_{ij}(x,t)  \xi_i \xi_j \leq \Lambda \sum_{d(i)=1} \xi_i^2 \end{equation}
for  a.e. $(x,t)$ and all $\xi \in \R^m$, for a suitable constant $\Lambda$. 
\begin{prop}\label{weak-harnack} Let $(\M,\mu,d)$ be a $2$-admissible structure. For a bounded open subset $\Omega\subset \M$, let $u$ be a non-negative, weak supersolution to \eqref{linear} in an open set containing the cylinder $\Omega\times [0,T_0]$ and assume that conditions \eqref{unifellip-l} are satisfied. 
For any  subcylinder $Q_{3\rho}=B(\bar x,3\rho)\times(\bar t-9\rho^2,\bar t)\subset Q$  there exists a constant $C>0$ depending on
$C_D$, $C_L$, $C_P$, the structure conditions \eqref{structure} and on $\rho$ such that
\begin{equation}
\frac{1}{|Q^-|} \int_{Q^-} w \ dx dt \le C(\inf_{Q^+} w +\sup_{Q^+} |g| \rho^2),
\end{equation}
with $Q^+, Q^-$ as defined in \eqref{cyl-1}.
\end{prop}

%% file: survey-cc.bbl
\def\cprime{$'$} \def\cprime{$'$} \def\cprime{$'$} \def\cprime{$'$}
\begin{thebibliography}{10}

\bibitem{CCM3}


\bibitem{ACCN}
{\sc Avelin, B., Capogna, L., Citti, G., and Nystr{\"o}m, K.}
\newblock Harnack estimates for degenerate parabolic equations modeled on the
  subelliptic {$p$}-{L}aplacian.
\newblock {\em Adv. Math. 257\/} (2014), 25--65.

\bibitem{baloghrickly}
{\sc Balogh, Z.~M., and Rickly, M.}
\newblock Regularity of convex functions on {H}eisenberg groups.
\newblock {\em Ann. Sc. Norm. Super. Pisa Cl. Sci. (5) 2}, 4 (2003), 847--868.

\bibitem{CCgeometry:Risler}
{\sc Bella$\ddot{\text{i}}$che, A., and Risler}, Eds.
\newblock {\em Sub-Riemannian Geometry}, vol.~144 of {\em Progress in
  Mathematics}.
\newblock Birkh$\ddot{\text{a}}$user, Basel, 1996.

\bibitem{bieske}
{\sc Bieske, T.}
\newblock On $\infty$-harmonic functions on the {H}eisenberg group.
\newblock {\em Comm. Partial Differential Equations 3-4}, 27 (2002), 727--761.

\bibitem{bieske2}
{\sc Bieske, T.}
\newblock Comparison principle for parabolic equations in the {H}eisenberg
  group.
\newblock {\em Electron.\ J. Differential Equations\/} (2005), No.\ 95, 11 pp.\
  (electronic).

\bibitem{birolimosco}
{\sc Biroli, M., and Mosco, U.}
\newblock Sobolev and isoperimetric inequalities for {D}irichlet forms on
  homogeneous spaces.
\newblock {\em Atti Accad.\ Naz.\ Lincei Cl.\ Sci.\ Fis.\ Mat.\ Natur.\ Rend.\
  Lincei (9) Mat.\ Appl. 6}, 1 (1995), 37--44.

\bibitem{BONFI}
{\sc Bonfiglioli, A., Lanconelli, E., and Uguzzoni, F.}
\newblock Uniform {G}aussian estimates for the fundamental solutions for heat
  operators on {C}arnot groups.
\newblock {\em Adv. Differential Equations 7\/} (2002), 1153--1192.

\bibitem{BLU-2004}
{\sc Bonfiglioli, A., Lanconelli, E., and Uguzzoni, F.}
\newblock Fundamental solutions for non-divergence form operators on stratified
  groups.
\newblock {\em Trans. Amer. Math. Soc. 356}, 7 (2004), 2709--2737.

\bibitem{BLU}
{\sc Bonfiglioli, A., Lanconelli, E., and Uguzzoni, F.}
\newblock {\em Stratified {L}ie {G}roups and {P}otential {T}heory for their
  {S}ub-{L}aplacians}.
\newblock Springer, 2007.

\bibitem{BramantiBrandolini}
{\sc Bramanti, M., and Brandolini, L.}
\newblock Schauder estimates for parabolic nondivergence operators of
  {H}\"ormander type.
\newblock {\em J. Differential Equations 234}, 1 (2007), 177--245.

\bibitem{BBLU}
{\sc Bramanti, M., Brandolini, L., Lanconelli, E., and Uguzzoni, F.}
\newblock Non-divergence equations structured on h\"{o}rmander vector fields:
  heat kernels and harnack inequalities.
\newblock {\em Mem. Amer. Math. Soc. 204}, 961 (2010).

\bibitem{Bu}
{\sc Buser, P.}
\newblock Isospectral riemann surfaces.
\newblock {\em Ann. Inst. Fourier (Grenoble) 36\/} (1986), 167--192.

\bibitem{CC}
{\sc Capogna, L., and Citti, G.}
\newblock Generalized mean curvature flow in {C}arnot groups.
\newblock {\em Comm. Partial Differential Equations 34}, 7-9 (2009), 937--956.

\bibitem{ccm1}
{\sc Capogna, L., Citti, G., and Manfredini, M.}
\newblock Regularity of non-characteristic minimal graphs in the {H}eisenberg
  group {$\Bbb H^1$}.
\newblock {\em Indiana Univ. Math. J. 58}, 5 (2009), 2115--2160.

\bibitem{CCM2}
{\sc Capogna, L., Citti, G., and Manfredini, M.}
\newblock Smoothness of lipschitz minimal intrinsic graphs in heisenberg groups
  $\mathbb{H}^n$, $n>1$.
\newblock {\em Crelle's Journal 648\/} (2010), 75Ð110.

\bibitem{MR3108875}
{\sc Capogna, L., Citti, G., and Manfredini, M.}
\newblock Uniform {G}aussian bounds for subelliptic heat kernels and an
  application to the total variation flow of graphs over {C}arnot groups.
\newblock {\em Anal. Geom. Metr. Spaces 1\/} (2013), 255--275.

\bibitem{CCM4}
{\sc Capogna, L., Citti, G., and Manfredini, M.}
\newblock Regularity of mean curvature flow of graphs on lie groups free up to
  step two.
\newblock {\em Preprint\/} (2015).

\bibitem{CCR}
{\sc Capogna, L., Citti, G., and Rea, G.}
\newblock A subelliptic analogue of {A}ronson-{S}errin's {H}arnack inequality.
\newblock {\em Math. Ann. 357}, 3 (2013), 1175--1198.

\bibitem{cdpt:survey}
{\sc Capogna, L., Danielli, D., Pauls, S., and Tyson, J.}
\newblock {\em An introduction to the {H}eisenberg group and the
  sub-{R}iemannian isoperimetric problem}, vol.~259 of {\em Progress in
  Mathematics}.
\newblock Birkh\"auser Verlag, Basel, 2007.

\bibitem{ch:minimal}
{\sc Cheng, J.-H., and Hwang, J.-F.}
\newblock Properly embedded and immersed minimal surfaces in the {H}eisenberg
  group.
\newblock {\em Bull.\ Austral.\ Math.\ Soc. 70}, 3 (2004), 507--520.

\bibitem{MR3216825}
{\sc Cheng, J.-H., and Hwang, J.-F.}
\newblock Uniqueness of generalized {$p$}-area minimizers and integrability of
  a horizontal normal in the {H}eisenberg group.
\newblock {\em Calc. Var. Partial Differential Equations 50}, 3-4 (2014),
  579--597.

\bibitem{chmy:minimal}
{\sc Cheng, J.-H., Hwang, J.-F., Malchiodi, A., and Yang, P.}
\newblock Minimal surfaces in pseudohermitian geometry.
\newblock {\em Ann.\ Sc.\ Norm.\ Super.\ Pisa Cl.\ Sci.\ (5) 4}, 1 (2005),
  129--177.

\bibitem{chmy-2}
{\sc Cheng, J.-H., Hwang, J.-F., Malchiodi, A., and Yang, P.}
\newblock A {C}odazzi-like equation and the singular set for {$C^1$} smooth
  surfaces in the {H}eisenberg group, 2012.

\bibitem{chy}
{\sc Cheng, J.-H., Hwang, J.-F., and Yang, P.}
\newblock Existence and uniqueness for {$p$}-area minimizers in the
  {H}eisenberg group.
\newblock {\em Math. Ann. 337}, 2 (2007), 253--293.

\bibitem{MR2481053}
{\sc Cheng, J.-H., Hwang, J.-F., and Yang, P.}
\newblock Regularity of {$C^1$} smooth surfaces with prescribed {$p$}-mean
  curvature in the {H}eisenberg group.
\newblock {\em Math. Ann. 344}, 1 (2009), 1--35.

\bibitem{CMY}
{\sc Cheng, J.-H., Malchiodi, A., and Yang, P.}
\newblock Isoperimetric domains in homogeneous three-manifolds and the
  isoperimetric constant of the heisenberg group $h^1$.
\newblock preprint, 2015.

\bibitem{CiMa-F}
{\sc Citti, G., and Manfredini, M.}
\newblock Uniform estimates of the fundamental solution for a family of
  hypoelliptic operators.
\newblock {\em Potential Anal. 25}, 2 (2006), 147--164.

\bibitem{cw:1971}
{\sc Coifman, R.~R., and Weiss, G.}
\newblock {\em Analyse harmonique non-commutative sur certains espaces
  homogenes. (French) \`Etude de certaines int\'egrales singuli\`eres}.
\newblock Lecture Notes in Mathematics, Vol. 242. Springer-Verlag, 1971.

\bibitem{CorwinGreenleaf}
{\sc Corwin, L., and Greenleaf, F.~P.}
\newblock {\em Representations of nilpotent {L}ie groups and their
  applications, {P}art {I}: {B}asic theory and examples}.
\newblock Cambridge Studies in Advanced Mathematics. Cambridge University
  Press, Cambridge, 1990.

\bibitem{dgn:minimal}
{\sc Danielli, D., Garofalo, N., and Nhieu, D.-M.}
\newblock Sub-{R}iemannian calculus on hypersurfaces in {C}arnot groups.
\newblock {\em Adv. Math. 215}, 1 (2007), 292--378.

\bibitem{MR2472175}
{\sc Danielli, D., Garofalo, N., Nhieu, D.~M., and Pauls, S.~D.}
\newblock Instability of graphical strips and a positive answer to the
  {B}ernstein problem in the {H}eisenberg group {$\Bbb H^1$}.
\newblock {\em J. Differential Geom. 81}, 2 (2009), 251--295.

\bibitem{MR2648078}
{\sc Danielli, D., Garofalo, N., Nhieu, D.-M., and Pauls, S.~D.}
\newblock The {B}ernstein problem for embedded surfaces in the {H}eisenberg
  group {$\Bbb H^1$}.
\newblock {\em Indiana Univ. Math. J. 59}, 2 (2010), 563--594.

\bibitem{de:1981}
{\sc Debiard, A.}
\newblock Espaces $h^p$ au dessus de l'espace hermitien hyperbolique de $c^n$
  ($n>1$).
\newblock {\em J. Funct. Anal.}, 40 (1981), 185--265.

\bibitem{PolidoroDiFrancesco}
{\sc Di~Francesco, M., and Polidoro, S.}
\newblock Schauder estimates, {H}arnack inequality and {G}aussian lower bound
  for {K}olmogorov-type operators in non-divergence form.
\newblock {\em Adv. Differential Equations 11}, 11 (2006), 1261--1320.

\bibitem{DDR}
{\sc Dirr, N., Dragoni, F., and Von~Renesse, M.}
\newblock Evolution by mean curvature flow in sub-riemannian geometries.
\newblock {\em Communications on Pure and Applied Mathematics\/} (2010),
  307--326.

\bibitem{MFL}
{\sc Ferrari, F., Manfredi, J., and Liu, Q.}
\newblock On the horizontal mean curvature flow for axisymmetric surfaces in
  the heisenberg group.
\newblock {\em Commun. Contemp. Math.}, 3 (2014).

\bibitem{fol:1975}
{\sc Folland, G.~B.}
\newblock Subelliptic estimates and function spaces on nilpotent {L}ie groups.
\newblock {\em Ark. Mat. 2}, 13 (1975), 161--207.

\bibitem{MR753153}
{\sc Franchi, B., and Lanconelli, E.}
\newblock H\"older regularity theorem for a class of linear nonuniformly
  elliptic operators with measurable coefficients.
\newblock {\em Ann. Scuola Norm. Sup. Pisa Cl. Sci. (4) 10}, 4 (1983),
  523--541.

\bibitem{FSS1}
{\sc Franchi, B., Serapioni, R., and Serra~Cassano, F.}
\newblock Meyers-{S}errin type theorems and relaxation of variational integrals
  depending on vector fields.
\newblock {\em Houston J. Math. 22}, 4 (1996), 859--890.

\bibitem{frss:embedding}
{\sc Franchi, B., Serapioni, R., and Serra-Cassano, F.}
\newblock Approximation and imbedding theorems for weighted sobolev spaces
  associated with lipschitz continuous vector fields.
\newblock {\em Boll. Un. Mat. Ital. B 7}, 11 (1997), 83--117.

\bibitem{Friedman}
{\sc Friedman, A.}
\newblock {\em Partial differential equations of parabolic type}.
\newblock Prentice-Hall, Inc., Englewood Cliffs, N.J., 1964.

\bibitem{Friedrichs}
{\sc Friedrichs, K.~O.}
\newblock The identity of weak and strong extensions of differential operators.
\newblock {\em Trans. Amer. Math. Soc.\/} (1944), 132--151.

\bibitem{MR2979606}
{\sc Galli, M., and Ritor{\'e}, M.}
\newblock Existence of isoperimetric regions in contact sub-{R}iemannian
  manifolds.
\newblock {\em J. Math. Anal. Appl. 397}, 2 (2013), 697--714.

\bibitem{GarofaloNhieu:lip1998}
{\sc Garofalo, N., and Nhieu, D.~M.}
\newblock Lipschitz continuity, global smooth approximations and exten- sion
  theorems for sobolev functions in carnot-carathodory spaces.
\newblock {\em J. Anal. Math. 74\/} (1998), 67--97.

\bibitem{Ge}
{\sc Ge, Z.}
\newblock Collapsing {R}iemannian metrics to {C}arnot-{C}arath\'eodory metrics
  and {L}aplacians to sub-{L}aplacians.
\newblock {\em Canad. J. Math. 45}, 3 (1993), 537--553.

\bibitem{grig}
{\sc Grigor'yan, A.~A.}
\newblock The heat equation on non-compact riemannian manifolds.
\newblock {\em Mat. Sb. (1) 182\/} (1991).

\bibitem{GutierrezLanconelli}
{\sc Guti{\'e}rrez, C.~E., and Lanconelli, E.}
\newblock Maximum principle, nonhomogeneous {H}arnack inequality, and
  {L}iouville theorems for {$X$}-elliptic operators.
\newblock {\em Comm. Partial Differential Equations 28}, 11-12 (2003),
  1833--1862.

\bibitem{hak:sobolev2}
{\sc Haj{\l}asz, P., and Koskela, P.}
\newblock Sobolev met {P}oincar\'e.
\newblock {\em Memoirs Amer.\ Math.\ Soc. 145}, 688 (2000).

\bibitem{pau:cmc-carnot}
{\sc Hladky, R.~K., and Pauls, S.~D.}
\newblock Constant mean curvature surfaces in sub-{R}iemannian geometry.
\newblock {\em J. Differential Geom. 79}, 1 (2008), 111--139.

\bibitem{MR2579306}
{\sc Hladky, R.~K., and Pauls, S.~D.}
\newblock Minimal surfaces in the roto-translation group with applications to a
  neuro-biological image completion model.
\newblock {\em J. Math. Imaging Vision 36}, 1 (2010), 1--27.

\bibitem{MR2609016}
{\sc Hurtado, A., Ritor{\'e}, M., and Rosales, C.}
\newblock The classification of complete stable area-stationary surfaces in the
  {H}eisenberg group {$\Bbb H^1$}.
\newblock {\em Adv. Math. 224}, 2 (2010), 561--600.

\bibitem{jer:poincare}
{\sc Jerison, D.}
\newblock The {P}oincar{\'e}\ inequality for vector fields satisfying
  {H}{\"o}rmander's condition.
\newblock {\em Duke Math.\ J. 53\/} (1986), 503--523.

\bibitem{MR865430}
{\sc Jerison, D.~S., and S{\'a}nchez-Calle, A.}
\newblock Estimates for the heat kernel for a sum of squares of vector fields.
\newblock {\em Indiana Univ. Math. J. 35}, 4 (1986), 835--854.

\bibitem{Kilpelainen}
{\sc Kilpel{\"a}inen, T.}
\newblock Smooth approximation in weighted {S}obolev spaces.
\newblock {\em Comment. Math. Univ. Carolin. 38}, 1 (1997), 29--35.

\bibitem{kor:1983}
{\sc Kor\'anyi, A.}
\newblock {\em Geometric aspects of analysis on the Heisenberg group}.
\newblock Ist. Naz. Alta Mat. Francesco Severi, 1983, pp.~209--258.

\bibitem{kor:heisenberg}
{\sc Kor{\'a}nyi, A.}
\newblock Geometric properties of {H}eisenberg-type groups.
\newblock {\em Adv.\ in Math. 56\/} (1985), 28--38.

\bibitem{Krylov}
{\sc Krylov, N.~V.}
\newblock H\"older continuity and {$L_p$} estimates for elliptic equations
  under general {H}\"ormander's condition.
\newblock {\em Topol. Methods Nonlinear Anal. 9}, 2 (1997), 249--258.

\bibitem{Kusuoka}
{\sc Kusuoka, S., and Stroock, D.}
\newblock Long time estimates for the heat kernel associated with a uniformly
  subelliptic symmetric second order operator.
\newblock {\em Ann. of Math 127}, 2 (1988), 165--189.

\bibitem{LanMor}
{\sc Lanconelli, E., and Morbidelli, D.}
\newblock On the poincare inequality for vector fields.
\newblock {\em Ark. Mat. 38\/} (2000), 327--342.

\bibitem{lms}
{\sc Lu, G., Manfredi, J.~J., and Stroffolini, B.}
\newblock Convex functions on the {H}eisenberg group.
\newblock {\em Calc. Var. Partial Differential Equations 19}, 1 (2004), 1--22.

\bibitem{Lunardi}
{\sc Lunardi, A.}
\newblock {\em Analytic semigroups and optimal regularity in parabolic
  problems}.
\newblock Modern Birkh\"auser Classics. Birkh\"auser/Springer Basel AG, Basel,
  1995.
\newblock [2013 reprint of the 1995 original] [MR1329547].

\bibitem{magnani:convex}
{\sc Magnani, V.}
\newblock Lipschitz continuity, {A}leksandrov theorem, and characterizations
  for {H}-convex functions.
\newblock {\em Math.\ Ann. 334\/} (2006), 199--233.

\bibitem{MR2512155}
{\sc Manfredini, M.}
\newblock Uniform {S}chauder estimates for regularized hypoelliptic equations.
\newblock {\em Ann. Mat. Pura Appl. (4) 188}, 3 (2009), 417--428.

\bibitem{mitchell}
{\sc Mitchell, J.}
\newblock On {C}arnot-{C}arath\'eodory metrics.
\newblock {\em J. Differential Geom. 21}, 1 (1985), 35--45.

\bibitem{Montgomery}
{\sc Montgomery, R.}
\newblock {\em Survey of singular geodesics}.
\newblock Vol.~144 of Bella$\ddot{\text{i}}$che and Risler
  \cite{CCgeometry:Risler}, 1996, pp.~325--339.

\bibitem{Montgomery:book}
{\sc Montgomery, R.}
\newblock {\em A tour of sub-{R}iemannian geometries, their geodesics and
  applications.}
\newblock No.~91 in Mathematical Surveys and Monographs. American Mathematical
  Society, 2002.

\bibitem{monti-tesi}
{\sc Monti, R.}
\newblock Distances, boundaries and surface measures in carnot-caratheodory
  spaces.
\newblock {\em Ph.D Thesis, Universit\'a degli studi di Trento\/} (2001).

\bibitem{Monti-Rickly05}
{\sc Monti, R., and Rickly, M.}
\newblock Geodetically convex sets in the {H}eisenberg group.
\newblock {\em J. Convex Anal. 12}, 1 (2005), 187--196.

\bibitem{NSW}
{\sc Nagel, A., Stein, E.~M., and Wainger, S.}
\newblock Balls and metrics defined by vector fields. {I}. {B}asic properties.
\newblock {\em Acta Math. 155}, 1-2 (1985), 103--147.

\bibitem{Pauls:minimal}
{\sc Pauls, S.~D.}
\newblock Minimal surfaces in the {H}eisenberg group.
\newblock {\em Geom.\ Dedicata 104\/} (2004), 201--231.

\bibitem{pau:obstructions}
{\sc Pauls, S.~D.}
\newblock {$H$}-minimal graphs of low regularity in {$\Bbb H\sp 1$}.
\newblock {\em Comment. Math. Helv. 81}, 2 (2006), 337--381.

\bibitem{MR2898770}
{\sc Ritor{\'e}, M.}
\newblock A proof by calibration of an isoperimetric inequality in the
  {H}eisenberg group {${\Bbb H}^n$}.
\newblock {\em Calc. Var. Partial Differential Equations 44}, 1-2 (2012),
  47--60.

\bibitem{RR}
{\sc Ritor{\'e}, M., and Rosales, C.}
\newblock Rotationally invariant hypersurfaces with constant mean curvature in
  the {H}eisenberg group {$\Bbb H\sp n$}.
\newblock {\em J. Geom. Anal. 16}, 4 (2006), 703--720.

\bibitem{RR1}
{\sc Ritor{\'e}, M., and Rosales, C.}
\newblock Area-stationary surfaces in the {H}eisenberg group {$\Bbb H^1$}.
\newblock {\em Adv. Math. 219}, 2 (2008), 633--671.

\bibitem{Roth:Stein}
{\sc Rothschild, L.~P., and Stein, E.~M.}
\newblock Hypoelliptic differential operators and nilpotent groups.
\newblock {\em Acta Math. 137}, 3-4 (1976), 247--320.

\bibitem{MR1771424}
{\sc Rumin, M.}
\newblock Sub-{R}iemannian limit of the differential form spectrum of contact
  manifolds.
\newblock {\em Geom. Funct. Anal. 10}, 2 (2000), 407--452.

\bibitem{SC}
{\sc Saloff-Coste, L.}
\newblock A note on {P}oincar\'e, {S}obolev, and {H}arnack inequalities.
\newblock {\em Internat. Math. Res. Notices 1992}, 2 (1992), 27--38.

\bibitem{Serra}
{\sc Serra~Cassano, F.}
\newblock On the local boundedness of certain solutions for a class of
  degenerate elliptic equations.
\newblock {\em Boll. Un. Mat. Ital. B (7) 10}, 3 (1996), 651--680.

\bibitem{MR1387522}
{\sc Sturm, K.~T.}
\newblock Analysis on local {D}irichlet spaces. {III}. {T}he parabolic
  {H}arnack inequality.
\newblock {\em J. Math. Pures Appl. (9) 75}, 3 (1996), 273--297.

\bibitem{wang:aronsson}
{\sc Wang, C.-Y.}
\newblock The aronsson equation for absolute minimizers of l-infinity
  functionals associated with vector fields satisfying hormanders condition.
\newblock {\em Transactions of American Mathematical Society\/}.
\newblock to appear.

\bibitem{wang:convex}
{\sc Wang, C.~Y.}
\newblock Viscosity convex functions on {C}arnot groups.
\newblock {\em Proc.\ Amer.\ Math.\ Soc. 133}, 4 (2005), 1247--1253
  (electronic).

\bibitem{Xu}
{\sc Xu, C.~J.}
\newblock Regularity for quasilinear second-order subelliptic equations.
\newblock {\em Comm. Pure Appl. Math. 45}, 1 (1992), 77--96.

\end{thebibliography}
